\documentclass[12pt]{amsart}

\usepackage{amsfonts}
\usepackage{latexsym}
\usepackage{amssymb}
\usepackage{amsthm}
\usepackage{amsmath,amssymb,amsfonts,latexsym,graphicx}
\usepackage{amsmath}
\usepackage{graphicx}
\usepackage{psfrag}
\usepackage{subfigure}
\usepackage{color}

\numberwithin{equation}{section} \numberwithin{figure}{section}
\def\Re{\mathbb{R}}

\def\n{\noindent}

\def\d{\displaystyle}

\def\b#1{\overline#1}
\def\bga{\begin{array}}
\def\eda{\end{array}}

\def\De{\Delta}
\def\pt{\partial}

\def\ev{\equiv}

\def\gm{\gamma}

\def\sg{\sigma}

\def\bbf{\mathbf{f}}

\def\bu{\mathbf{u}}

\def\bA{\mathbf{A}}

\def\bF{\mathbf{F}}
\def\bG{\mathbf{G}}

\def\bP{\mathbf{P}}

\def\al{\alpha}
\def\td{\tilde}
\def\rw{\rightarrow}

\def\om{\omega}

\def\d{\displaystyle}
\def\dfr#1#2{\displaystyle{\frac{#1}{#2}}}

 \newtheorem{thm}{Theorem}[section]
 \newtheorem{cor}[thm]{Corollary}
 
 \newtheorem{prop}[thm]{Proposition}
 \theoremstyle{definition}
 
 \theoremstyle{remark}
 \newtheorem{rem}[thm]{Remark}

\include{abbr-notation}
\title{Thermodynamical Effects and High Resolution Methods for Compressible Fluid Flows}

\thanks{  }

\author{  Jiequan Li  and Yue Wang }
\address{
 Laboratory of Computational Physics, Institute of Applied Physics and Computational Mathematics, Beijing, 100088, China  }

\email{
  Jiequan Li:  li\_jiequan@iapcm.ac.cn; Yue Wang: wang\_yue@iapcm.ac.cn
}

\begin{document}



\begin{abstract}
One of the fundamental differences of compressible fluid flows from incompressible fluid flows is the involvement of  thermodynamics. This difference should be manifested  in the design of numerical methods and seems often be neglected in addition that the entropy inequality, as a conceptual derivative, is taken into account to reflect  irreversible processes and verified for some first order schemes.  In this paper, we refine  the GRP solver to illustrate how the thermodynamical variation is integrated into the design of high resolution methods for compressible fluid flows and demonstrate numerically the importance of thermodynamic effect in the resolution of strong waves. As a by-product, we show that the GRP solver works for generic equations of state, and is independent of technical arguments.

\end{abstract}
\maketitle



\pagestyle{myheadings} \thispagestyle{plain} \markboth{
   JIEQUAN LI and YUE WANG} { Thermodynamical Effects and High Resolution Methods for Compressible Fluid Flows  }

\section{Introduction}\label{sec1}

In the study of compressible fluid flows, the thermodynamical (Gibbs) relation
\begin{equation}
Tds =de -\dfr{p}{\rho^2} d\rho,
\label{Gibbs}
\end{equation}
has always the fundamental importance,
where  $T$ is the temperature, $s$ is the entropy, $e$ is the internal energy, $p$ is the pressure and $\rho$ is the density. In general, the internal energy includes the static energy, chemical energy in the field of combustion, and stress tensor in elasticity etc.
This relation  distinguishes compressible fluid flows from incompressible fluid ones.  The more compressible the flows are, the more dominant role it needs to play. This feature should be manifested in the design of numerical methods  in order to guarantee  resulting numerical solutions obey the same or at least an approximate analogue.
 \vspace{0.2cm}

Let us  recall the (first order) Godunov scheme  \cite{Godunov} for inviscid compressible Euler equations to   roughly inspect how the thermodynamics works for numerical methods  since it is the reference of any first order numerical schemes and has become the foundation of modern CFD. The  Godunov scheme  assumes,  as most first order finite volume schemes do,  that the physical state is uniform over each computational cell at each time step and the flow variation is described through the jumps of states across neighboring cell boundaries. The cellwise uniformity of the initial data implies that the resulting rarefaction waves emanating from the jumps are always isentropic. Moreover, the local  self-similarity of the solution implies that the entropy is constant along each cell boundary.  It turns out that no thermodynamical process is included in numerical fluxes. It does not matter if the thermodynamical effect is weak and the dynamical process is not severe, just as exhibited in many popular numerical examples in literature, e.g. \cite{Sod}.  However,  once the thermodynamical process becomes significant, an effective numerical scheme has to include the
thermodynamical variation. Otherwise, it might have some defects, as shown in \cite{Tang-Liu} for the problem of  large ratio of density or pressure.  This is insurmountable in the framework of first order schemes or higher order accurate schemes when first order numerical flux (Riemann, approximate Riemann) solvers are adopted.
\vspace{0.2cm}

In this paper, we will demonstrate how the thermodynamical effect is integrated into  high order accurate numerical schemes through the study of the generalized Riemann problem (GRP).   Given piecewise polynomials of high degree as initial data, all waves from the initial jumps at cell boundaries are curved and in particular rarefaction waves are no longer isentropic. The  initial entropy variation activates the interaction between the  kinematical and thermodynamical quantities.
As the waves are sufficiently strong, the interaction cannot be neglected and
  the entropy variation should be plugged into numerical fluxes in order to precisely characterize the thermodynamical process, as described in the GRP approach in \cite{Ben2006}.  We recognize by careful inspection on the resolution of the GRP that the thermodynamical quantities play  a fundamental role in the following sense:

\begin{enumerate}
\item[(i)] The expansion of waves can be characterized  in terms of local sound speed alone;
\vspace{0.2cm}

\item[(ii)]  The  entropy variation rate across rarefaction waves only depends on the local sound speed;

\vspace{0.2cm}

\item[(iii)] The interaction of kinematic  and thermodynamical quantities is strongly influenced by the entropy  variation.

\end{enumerate}
These explain why  GRP-based numerical schemes work well for problems under extreme conditions (e.g. high temperature  and  high pressure etc.) and the thermodynamics plays an important role in the design of high resolution schemes.
  \vspace{0.2cm}

As another purpose, this paper aims to refine the GRP solvers that were derived before, such as the original GRP solver in \cite{Ben-Artzi-1984,Ben-Artzi-2003}, the Eulerian version \cite{Ben2006, Ben2007} and high order extension \cite{Qian-2014, Tang-Yang, Wang-Wang}.  In the development of the GRP solver,  an important progress was made in \cite{Li-Chen} where the Riemann invariants were introduced so that the GRP solver can be extended to general hyperbolic balance laws.
The current contribution emphasizes that the GRP solver is  independent of technical analyses although the tricky  technique of "nonlinear geometric optics" is applied to the singularity point. It turns out that the technicality has nothing to do with any specific equation of state and readers are relieved of the tedious description of local characteristic coordinates in the previous versions of GRP solver in \cite{Ben-Artzi-2003, Ben2006, Ben2007} even though all ingredients are already cooked there.  Thus the  conclusion is applicable beyond gas dynamics.

\vspace{0.2cm}

In order to  keep the clarity of our presentation, we confine  our discussion in the finite volume framework, following van Leer's philosophy \cite{Leer}.   In Section 2, we simply summarize  numerical methods for compressible fluid flows and relate them to the generalized Riemann problem (GRP). In Section 3,  we refine the arguments on the resolution of rarefaction waves and highlight the role of thermodynamics. The technique of ``nonlinear geometric optics" is applied to measure how rarefaction waves expand and how the thermodynamics takes effect on the kinematic quantities.  In Section 4, we resolve shocks via the singularity tracking. The GRP solver is summarized in Section 5. In Section 6, a numerical demonstration is given to emphasize the importance of thermodynamic effect included in numerical fluxes. In Section 7, more remarks are presented about the GRP solver.
 All notations we use are put in Table I in Appendix.

\vspace{0.2cm}

 \section{Set-up of  high order schemes and the generalized Riemann problem}

We  write the flow equations in general form,
 \begin{equation}
 \bu_t +\bbf(\bu)_x =\bG(x,\bu),   \ \ \  t>0,
 \label{balance}
 \end{equation}
 where $\bbf(\bu)$ is the flux function of physical vector $\bu$, $\bG(x,\bu)$ is the source term, $x$ is the spatial variable, $t$ is the temporal variable. The prototype of \eqref{balance} is the compressible Euler equations with external forces or geometrical effects,
 \begin{equation}
 \bu=(\rho,\rho u, \rho E)^\top,\ \ \ \ \bbf(\bu) =(\rho u,\rho u^2+p, u(\rho E+p))^\top,
 \end{equation}
where the primitive variables $\rho$, $u$, $p$ are density, velocity and pressure, respectively;  the total energy consists of the kinematic energy $u^2/2$ and the internal energy $e$,  $E =\frac{u^2}2+e$, the internal energy $e$ is defined through the Gibbs relation \eqref{Gibbs}.

 \vspace{0.2cm}

A high order finite volume scheme for \eqref{balance} assumes piecewise polynomials of degree $k$  at each time step $t=t_n$  over each computational cell $I_j=(x_{j-\frac 12},x_{j+\frac 12})$,
\begin{equation}
\bu(x,t_n) =\bP_j^{k,n}(x), \ \ \  x\in I_j.
\label{data}
\end{equation}
where $x_j=j \De x$, $x_{j+\frac 12} =\frac 12(x_j+x_{j+1})$,  and $\De x$ is the spatial mesh size. Let $\De t$ be the time  step size  and satisfy the usual CFL constraint.
The numerical solution is updated in two steps.
\begin{enumerate}
\item[(i)] {\em Average advancing.}  We advance the solution average  of \eqref{balance} and \eqref{data}  to the next time step according to the formula
\begin{equation}
\begin{array}{l}
\bu_j^{n+1} =\bu_j^n -\dfr{\De t}{\De x} [\bF_{j+\frac 12}^* -\bF_{j-\frac 12}^*] + \De t \bG_j^*, \\[3mm]
 \bu_j^n =\dfr{1}{\De x} \int_{x_{j-\frac 12}}^{x_{j+\frac 12}}\bu(x,t_n)dx,
 \end{array}
\label{balance-scheme}
\end{equation}
where $\bF_{j+\frac 12}^*$ is the  numerical flux at  the  cell boundary $x=x_{j+\frac 12 }$  and  $\bG_j^*$ is the proper evaluation of the source term over the control volume $I_j\times (t_n,t_{n+1})$,  $t_{n+1} =t_n +\De t$.
\vspace{0.2cm}

\item[(ii)] {\em Projection of data.}  We project the solution $\bu(x,t_{n+1}-)$ to the space of piecewise polynomials $\bu(x,t_{n+1}) =\bP_j^{k,n+1}(x)$, $x\in I_j$.
\end{enumerate}
\vspace{0.2cm}

This paper focuses on the approximation of flux function,
\begin{equation}
 \ \ \ \bF_{j+\frac 12}^* \approx  \dfr{1}{\De t}\int_{t_n}^{t_{n+1}} \bbf(\bu(x_{j+\frac 12},t))dt.
\label{flux-num}
\end{equation}
Such an approximation depends on the resolution of the  generalized Riemann problem (GRP) at each singularity point $(x_{j+\frac 12}, t_n)$,
\begin{equation}
\begin{array}{l}
 \bu_t +\bbf(\bu)_x =\bG(x,\bu), \ \ \  x\in (x_{j-\frac 12},x_{j+\frac 32}),   \ \ \ t_n<t<t_{n+1}, \\[3mm]
 \bu(x,t_n) =\left\{
 \begin{array}{ll}
 \bP_j^{k,n}(x), \ \ \ &x\in I_j,\\
 \bP_{j+1}^{k,n}(x), & x\in I_{j+1}.
 \end{array}
 \right.
 \end{array}
 \label{GRP-1}
\end{equation}
We shift  $(x_{j+\frac 12},t_n)=(0,0)$ and  denote
\begin{equation}
\begin{array}{c}
 \bu_L =\bP_j^{k,n}(0_-),  \ \ \ \ \bu_R =\bP_{j+1}^{k,n}(0_+),\\[3mm]
 \bu_L' =\dfr{d}{dx}\bP_j^{k,n}(0_-), \ \ \  \bu_L' =\dfr{d}{dx}\bP_{j+1}^{k,n}(0_+).
\end{array}
 \end{equation}
The task of the  {\em GRP solver}  aims to evaluating the value $\bu(0,t)$ along the cell boundary $x=0$,
 \begin{equation}
 \bu_* =\lim_{t\rw 0+} \bu(0,t), \ \ \ \ \left(\dfr{\pt \bu}{\pt t}\right)_* =\lim_{t\rw 0+} \dfr{\pt \bu}{\pt t}(0,t).
 \end{equation}
 Once these are available, we evaluate
 \begin{equation}
 \bu(x_{j+\frac 12},t_n+\dfr{\De t}2) = \bu_* +\dfr{\De t} 2 \left(\dfr{\pt \bu}{\pt t}\right)_* +\mathcal{O}(\De t^2),
 \end{equation}
 and the flux across each interface $x=x_{j+\frac 12}$  is approximated as
 \begin{equation}
 \bF_{j+\frac 12}^* := \bbf( \bu(x_{j+\frac 12},t_n+\dfr{\De t}2) ),
 \end{equation}
 or
 \begin{equation}
  \bF_{j+\frac 12}^* := \bbf(\bu_*)+ \dfr{\De t}2\dfr{\pt\bbf(\bu_*)}{\pt \bu}\left(\dfr{\pt \bu}{\pt t}\right)_*,
 \end{equation}
 within second order accuracy. We remark that we are satisfied with the second order accurate GRP solver because it is sufficient to be a building block so as to construct multi-stage higher order schemes adopting the strategy in \cite{Li-Du}.    \vspace{0.2cm}

  Certainly, there are many approaches approximating the numerical flux through the technique of interpolation. The numerical example in Section \ref{sec-num}  shows  that any interpolation should reflect the thermodynamical effect precisely once the   effect becomes prominent. This is not a trivial task.  Instead of plausible interpolations, we want to see the exact information along the cell boundary so that  the numerical flux can be properly constructed.
 \vspace{0.2cm}

The resolution of the GRP is closely related to the associated Riemann problem with the asymptotic property, as sketched in the following proposition.

\begin{prop} Let $\bu(x,t)$ be the solution of \eqref{GRP-1}, and $\bu^A(x,t) =R^A(x/t; \bu_L,\bu_R)$ be the solution to the associated Riemann problem,
\begin{equation}
\begin{array}{l}
 \bu^A_t +\bbf(\bu^A)_x =0, \ \ \ x\in \Re, \ \ \   t>0,\\[3mm]
\bu^A(x,0) =\left\{
\begin{array}{ll}
\bu_L, \ \ \ &x<0,\\[3mm]
\bu_R , &x>0.
\end{array}
\right.
\end{array}
\label{RP}
\end{equation}
Then at the singularity point $(0,0)$, there holds
\begin{equation}
\lim_{t\rw 0+} \bu(\al t, t) =R^A(\al; \bu_L, \bu_R).
\end{equation}
for any $\al\in \Re$.
\end{prop}
\vspace{0.2cm}

\begin{figure}[!htb]
\centering \subfigure[Wave pattern for the GRP. The initial data
$\bu(x,0)= \bu_L+x\bu_L'$ for $x<0$ and $\bu(x,0)=\bu_R+\bu_R' x$ for $x>0$. ]{
\psfrag{0}{$0$}
\psfrag{x}{$x$}\psfrag{t}{$t$}\psfrag{betal}{\small
$\beta=\beta_L$}\psfrag{ul}{$\bu_L$}\psfrag{al2}{$\bar{\bar\al}$}
\psfrag{shock}{shock}\psfrag{rarefaction}{rarefaction}\psfrag{uminus}{$\bu_-(x,t)$}\psfrag{uplus}{$\bu_+(x,t)$}
\psfrag{u1}{$\bu_{1*}$}\psfrag{u2}{$\bu_{2*}$}\psfrag{al=albar}{$\al=\bar\al$}\psfrag{xbar}{$\bar
\al$ }\psfrag{al=2albar}{$\al=\bar{\bar \al}$}\psfrag{ur}{$\bu_R$}
\psfrag{contact}{contact}\psfrag{betastar}{$\beta=\beta_*$}
\psfrag{ul}{$\bu_L$}\psfrag{ur}{$\bu_R$}\psfrag{ustar}{$\bu_*$}
\includegraphics[clip, width=9cm]{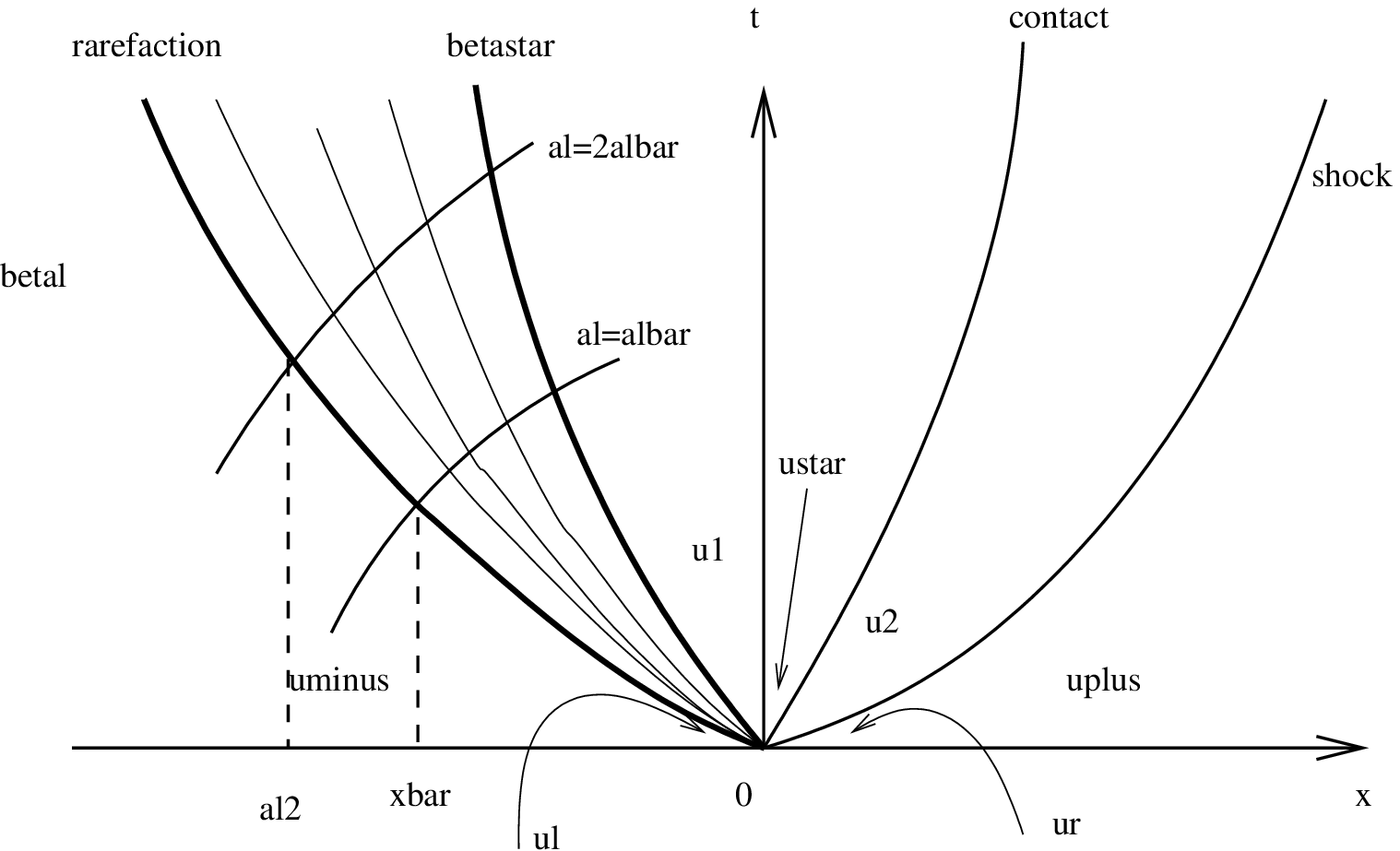} }
\hspace{1cm}

 \subfigure[Wave pattern for the associated Riemann problem]{
\psfrag{x}{$x$}\psfrag{t}{$t$}\psfrag{betal}{\small
$\beta=\beta_L$}
\psfrag{shock}{shock}\psfrag{rarefaction}{rarefaction}\psfrag{uminus}{$\bu_L$}\psfrag{uplus}{$\bu_R$}
\psfrag{u1}{$\bu_{1*}$}\psfrag{u2}{$\bu_{2*}$}\psfrag{al=albar}{$\al=\bar\al$}\psfrag{xbar}{$\bar
\al$ }\psfrag{al=2albar}{$\al=\bar{\bar \al}$}
\psfrag{contact}{contact}\psfrag{betastar}{$\beta=\beta_*$}\psfrag{ustar}{$\bu_*$}
\psfrag{0}{$0$}\psfrag{al2}{$\bar{\bar \al}$}
\includegraphics[width=9cm]{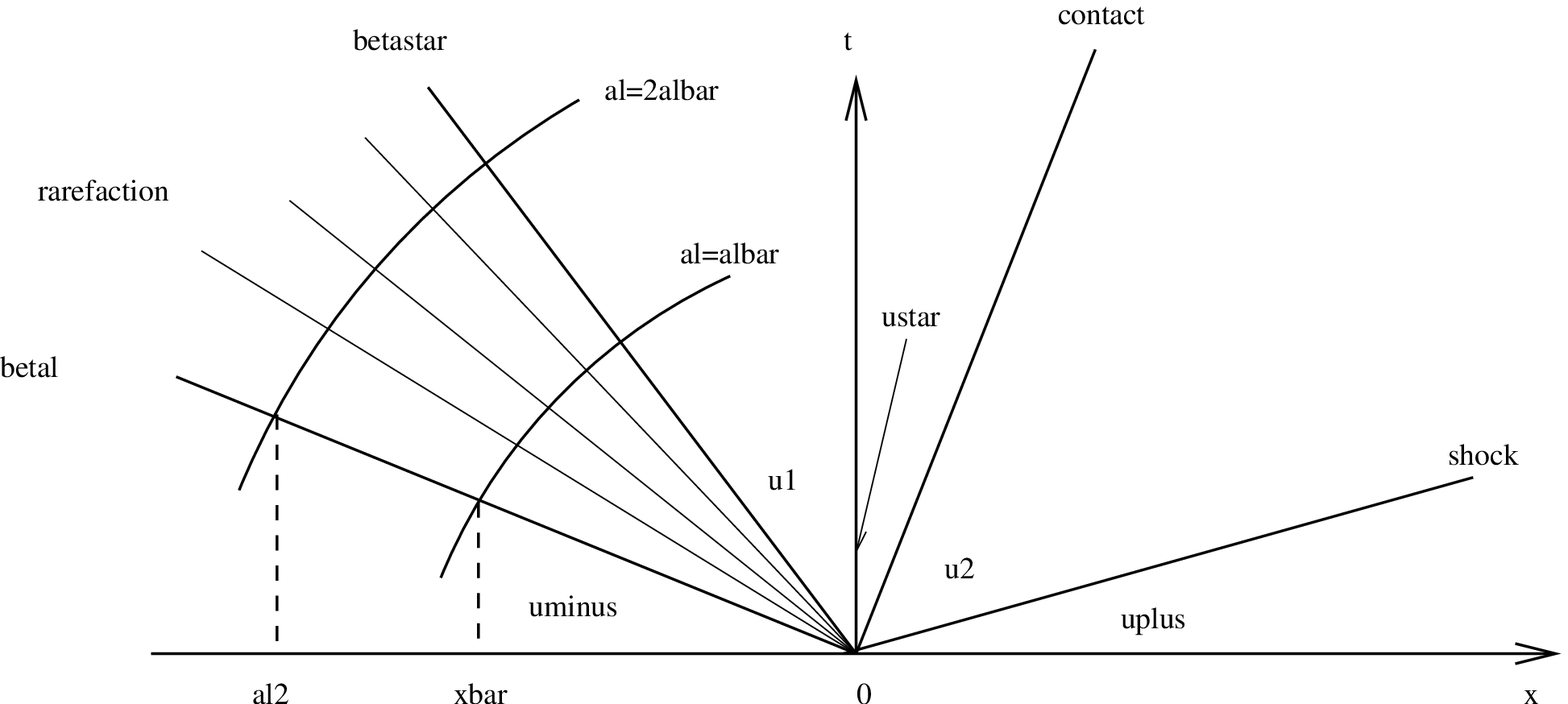}}
\caption[small]{Typical wave pattern for the generalized Riemann problem}
\label{Fig-wave}
\end{figure}

With this proposition, we have
\begin{equation}
\bu_* =R^A(0; \bu_L,\bu_R).
\end{equation}
Although the Riemann solution $\bu^A(x,t)$ is self-similar,  all waves in the solution of GRP are curved. In particular,  the curved rarefaction wave is no longer isentropic, and the initial variation of entropy should be carefully quantified, which is the key to plug the thermodynamical effect into the numerical flux.

\vspace{0.2cm}

  Without loss of generality,  we assume that a rarefaction wave emanating from $(0,0)$ moves to the left,  a shock moves to the right and the $t$-axis is located inside the intermediate region, as shown in Figure \ref{Fig-wave} \footnote{If all waves move to one side of $x=0$, the GRP is solved upwind.}.  The rarefaction wave is associated with the characteristic field $u-c$, and the transversal characteristic field is $u+c$, where $u$ is the flow velocity  and $c$ is the local sound speed, $c^2 =\frac{\pt p(\rho, s)}{\pt \rho}$.


 \section{ Resolution of rarefaction waves }

In this section we will refine the arguments on the resolution of rarefaction waves in \cite{Ben2006} although all ingredients are already cooked there. The point here is to emphasize that all arguments are independent of the specific equation of state and highlight the thermodynamic effect.
\vspace{0.2cm}

 \subsection{The  thermodynamical effect on the  expansion of rarefaction waves}

In order to understand how the rarefaction wave expands. we resolve the singularity at the origin using a technique of ``nonlinear geometrical optics" and come to conclusion that the local sound speed determines the expansion.  The  involvement of characteristic coordinates is crucial, but the conclusion is independent of technical arguments.
 \vspace{0.2cm}

  Let $\al(x,t)=C_1$ and $\beta(x,t)=C_2$ be the integral curves, respectively, of
  \begin{equation}
  \dfr{dx}{dt} =u+c,  \ \ \ \dfr{dx}{dt} =u-c.
  \end{equation}
 Away from the vacuum and the singularity point $(0,0)$,  there is a one-to-one correspondence $(x, t)\rightarrow(\alpha,\beta)$ such that,
    \begin{align}
  &\frac{\pt \alpha}{\pt x}dx+\frac{\pt \alpha}{\pt t}d t=0, \notag\\
  &\frac{\pt \beta}{\pt x}dx+\frac{\pt \beta}{\pt t}d t=0. \notag
  \end{align}
 Regarding the inverse $(\al,\beta)\rw (x,t)$,  we have
  \begin{align}
  \dfr{\pt x}{\pt \alpha}=(u-c)\dfr{\pt t}{\pt \alpha},  \  \  \ \dfr{\pt x}{\pt \beta}=(u+c)\dfr{\pt t}{\pt \beta}.
  \label{x-tau}
  \end{align}
  In the following discussion, we often use both pairs of  independent  variables $(x,t)$ or $(\al,\beta)$. For example, as we use $(\al, \beta)$ as independent variables, we allude to the fact that $(x,t)=(x(\al,\beta),t(\al,\beta))$, if no confusion is caused.
  \vspace{0.2cm}

  Thanks to the  asymptotics  of the GRP to the associated Riemann problem at the singularity point, we denote by $\beta_L$ the slope of the wave head, by $\beta$ the speed of the  wave speed inside the rarefaction wave, and by $\beta_*$ the speed of wave tail.   Then we have the following fact.
 \vspace{0.2cm}

\begin{prop}\label{prop-singularity} Consider the curved rarefaction wave associated with $u-c$ and denote by  $\Theta(\beta):=\frac{\pt t}{\pt \al}(0,\beta)$.  Then we have
\begin{equation}
\dfr{\Theta(\beta)}{\Theta(\beta_L)} = \exp\left[\int_{\beta_L}^\beta\dfr{1}{2c(0,\xi)}d\xi\right]=:\Pi(c; \beta,\beta_L).
\label{t-al-ratio1}
\end{equation}
In particular, for polytropic gases, we have
\begin{equation}
\dfr{\Theta(\beta)}{\Theta(\beta_L)} = \left(\dfr{c(0,\beta)}{c_L}\right)^{\frac{1}{2\mu^2}}, \ \ \ \mu^2=\frac{\gm-1}{\gm+1}.
\label{t-al-ratio2}
\end{equation}
\end{prop}

\begin{proof} From \eqref{x-tau}, we   obtain
\begin{equation}
\begin{array}{l}
\dfr{\pt^2 x}{\pt\al\pt \beta} =\dfr{\pt (u-c)}{\pt \beta}\dfr{\pt t}{\pt \al}+(u-c)\dfr{\pt^2 t}{\pt\al\pt\beta},\\[3mm]
\dfr{\pt^2 x}{\pt\al\pt \beta} =\dfr{\pt (u+c)}{\pt \al}\dfr{\pt t}{\pt \beta}+(u+c)\dfr{\pt^2 t}{\pt\al\pt\beta}.
\end{array}
\end{equation}
The subtraction of one from the other yields
\begin{equation}
2c\dfr{\pt^2 t}{\pt\al\pt\beta} =\dfr{\pt (u-c)}{\pt \beta}\dfr{\pt t}{\pt \al} -\dfr{\pt (u+c)}{\pt \al}\dfr{\pt t}{\pt \beta}.
\label{second-tau}
\end{equation}
We can set $t(0,\beta)=0$ and $\frac{\pt (u-c)}{\pt \beta}(0,\beta)=1$.  By noting that $\frac{\pt t}{\pt \beta}(0,\beta)\ev 0$ and
$\frac{\pt (u+c)}{\pt \al}$ is bounded, we obtain
\begin{equation}
\dfr{\pt}{\pt \beta}\Theta(\beta) =\dfr{1}{2c(0,\beta)} \Theta(\beta).
\label{Theta-eq}
\end{equation}
Integrating from $\beta_L$ to $\beta$ yields \eqref{t-al-ratio1}.  For the polytropic gases,  we have \eqref{t-al-ratio2}.
\end{proof}

\begin{rem} This proposition characterizes how the rarefaction wave expands near the singularity point  in terms of the characteristic coordinate $\al$. The local sound speed $c$ solely determines the degree of expansion.
\vspace{0.2cm}

\end{rem}

\subsection{ The rate of entropy variation across curved rarefaction waves}

Now we want to investigate how  the entropy varies across the curved rarefaction wave associated with $u-c$. The entropy function $s$ just keeps constant along any specific particle trajectory
\begin{equation}
\dfr{\pt s}{\pt t} +u \dfr{\pt s}{\pt x}=0.
\label{entropy-eq}
\end{equation}
But it does not mean that the entropy is uniform in the whole regime of the  curved rarefaction wave fan.
Given the initial state in \eqref{data}, the initial variation of the entropy is  known thanks to the Gibbs relation \eqref{Gibbs},
  \begin{equation}
  T_L s_{L}' = e_{L}'-\dfr{p_L}{\rho_L^2} \rho_{L}',
  \label{en-ini}
  \end{equation}
  where the superscript ``prime" represents the derivative of corresponding variables with respect to $x$, the subscript ``L" denotes the limit from the left.  Then the entropy variation across  the rarefaction wave
has a rate only depending on the local sound speed $c$.

  \begin{prop}  Across  the curved rarefaction wave associated with $u-c$, the entropy variation $s_x(0,\beta)$ in the neighborhood of the singularity point has the change rate,
  \begin{equation}
  \dfr{ s_x(0,\beta)}{s_{L}'} =\dfr{c_L}{c(0,\beta)} \exp\left[\int_{\beta_L}^\beta- \dfr{1}{c(0,\xi)}d\xi\right]=\dfr{c_L}{c(0,\beta)} \cdot \Pi^{-2}(c;\beta,\beta_L).
  \label{s-al}
  \end{equation}
  For polytropic gases, it becomes
  \begin{equation}
  \dfr{T s_x(0,\beta)}{T_Ls_{L}'}  =\left(\frac{c(0,\beta)}{c_L}\right)^{\frac{1}{\mu^2}+1}.
  \label{s-al-poly}
  \end{equation}
    \end{prop}

  \begin{proof}

  We rewrite the equation \eqref{entropy-eq} as,
  \begin{align}
  &\dfr{\pt s}{\pt t} + (u+c) \frac{\pt s}{\pt x} = c \frac{\pt s}{\pt x}, \notag\\
  &\dfr{\pt s}{\pt t} + (u-c) \frac{\pt s}{\pt x} =-c \frac{\pt s}{\pt x}. \notag
  \end{align}
In terms of  the characteristic coordinate $(\al,\beta)$, they become
   \begin{align}
  & \dfr{\pt s}{\pt \beta}=\dfr{\pt t}{\pt \beta}\dfr{\pt s}{\pt t}+\dfr{\pt x}{\pt \beta} \dfr{\pt s}{\pt x}=\dfr{\pt t}{\pt \beta}\left[\dfr{\pt s}{\pt t}+(u+c)\dfr{\pt s}{\pt x}\right]=\dfr{\pt t}{\pt \beta}\cdot \left(c \frac{\pt s}{\pt x}\right), \label{sbeta1}\\
  & \dfr{\pt s}{\pt \alpha}=\dfr{\pt t}{\pt \alpha}\dfr{\pt s}{\pt t}+\dfr{\pt x}{\pt \alpha} \dfr{\pt s}{\pt x}=\dfr{\pt t}{\pt \alpha}\left[\dfr{\pt s}{\pt t}+(u-c)\dfr{\pt s}{\pt x}\right]=\dfr{\pt t}{\pt \alpha}\cdot \left(-c \frac{\pt s}{\pt x}\right).\label{sbeta2}
  \end{align}
 In order to derive  \eqref{s-al}, it suffices to measure $\dfr{\pt s}{\pt\al}(0,\beta)$. For this purpose, we differentiate \eqref{sbeta1} with respect to $\al$,
 \begin{equation}
 \dfr{\pt^2 s}{\pt\al\pt \beta} = \dfr{\pt^2t}{\pt\al\pt \beta} \cdot \left(c \frac{\pt s}{\pt x}\right)+ \dfr{\pt t}{\pt \beta} \cdot \dfr{\pt}{\pt \al}\left(c \frac{\pt s}{\pt x}\right).
 \end{equation}
 Inserting \eqref{second-tau} and noting $\frac{\pt t}{\pt \beta}(0,\beta) \ev 0$, we obtain
 \begin{equation}
 \dfr{\pt^2 s}{\pt\al\pt \beta}(0,\beta)= \dfr{1}{2c(0,\beta)}\cdot  \dfr{\pt t}{\pt \al}(0,\beta)\cdot  \left(c \frac{\pt s}{\pt x}\right)(0,\beta).
 \end{equation}
 Substituting \eqref{sbeta2} into this equation yields
 \begin{equation}
  \dfr{\pt}{\pt\beta}\left(\dfr{\pt s}{\pt\al}(0,\beta)\right)=-\dfr{1}{2c(0,\beta)}\dfr{\pt s}{\pt \al}(0,\beta).
 \end{equation}
 This is an ODE for $\frac{\pt s}{\pt \al}(0,\beta)$. We integrate it from $\beta_L$ to $\beta$ to get
 \begin{equation}
 \dfr{\pt s}{\pt \al}(0,\beta)= \dfr{\pt s}{\pt \al}(0,\beta_L)\exp\left[\int_{\beta_L}^\beta- \dfr{1}{2c(0,\xi)}d\xi\right].
 \end{equation}
 We go back to the frame of $(x,t)$-coordinate, by using \eqref{sbeta2}, to obtain
 \begin{equation}
 \Theta(\beta) \left(-c \frac{\pt s}{\pt x}\right)(0,\beta) =\Theta(\beta_L) \left(-c \frac{\pt s}{\pt x}\right)(0,\beta_L)\cdot \exp\left[\int_{\beta_L}^\beta- \dfr{1}{2c(0,\xi)}d\xi\right].
 \end{equation}
 This is just \eqref{s-al}. Specified to the polytropic gases, we obtain \eqref{s-al-poly}.
\end{proof}
 \vspace{0.2cm}

 \begin{cor} \label{cir-s}  The instantaneous change of the entropy along the interface $x=0$ is
 \begin{equation}
\dfr{\pt s}{\pt t}(0,\beta_*) = -u_* s_L' \dfr{c_L}{c_*}\Pi^{-2}(c_*;\beta_*,\beta_L),
 \end{equation}
 if the interface is located inside the intermediate regime.
 \end{cor}
 \vspace{0.2cm}

 This corollary describes the entropy change along the cell interface and reflects the thermodynamical process.   The rate is independent of the specific form of the equation of state and even has nothing to do with geometric effects  possibly involved.  However, if certain more physical factors  are included in the Gibbs relation \eqref{Gibbs},  the variation of these factors is added into the  initial entropy variation through the internal energy variation, in view of \eqref{en-ini},  and propagates into the intermediate region.  In the next section, we will see how the entropy variation takes effect on numerical fluxes.

\vspace{0.2cm}

\subsection{ The interaction  of kinematics and thermodynamics}

In the design of numerical methods for compressible fluid flows, we often use characteristic quantities or
Riemann invariants.  We introduce
\begin{equation}
\psi =u +\int^\rho \dfr{c(s,\om)}{\om}d\om.
\end{equation}
This quantity is  named as  the ``Riemann invariant" associated with $u-c$ when the flow is isentropic, and describes the inherent relation between
the kinematic quantity $u$ and thermodynamical quantities
\begin{equation}
d\psi =du +\dfr{1}{\rho c}dp +K(\rho, s)ds, \ \ \  K(\rho, s) =-\dfr{1}{\rho c}\dfr{\pt p}{\pt s} +\int^\rho \dfr{1}{\om} \dfr{\pt c(\om,s)}{\pt s} d\om.
\end{equation}
For ideal gases, $K(\rho,s)=T/c$. However, we would like to rename it as a {\em `` kinematic-thermodynamic"} variable for non-isentropic case because it is not invariant over the whole rarefaction wave region (equivalently not invariant along the vector field defined by the eigenvector associated with $u-c$).  In general,  the quantity $\psi$ satisfies the equation of form,
  \begin{align}
  \dfr{\pt \psi}{\pt t} +(u+c) \frac{\pt \psi}{\pt x} = cK(\rho,s) \frac{\pt s}{\pt x}+G(x,t), \label{psi-eq}
  \end{align}
  where $G(x,t)$ results from external forces or geometrical effects.  Here we  consider the case that $G(x,t)\ev 0$ in order to see how the pure entropy variation takes effect.  The analysis is done in the same matter as that for the entropy variation.
  \vspace{0.2cm}

  Since the right hand side of \eqref{psi-eq} is already known from \eqref{s-al}, we denote it by $H(x,t)$ as a given function,
  \begin{equation}
  H(x,t) =cK(\rho, s)\dfr{\pt s}{\pt x}.
  \end{equation}
  We write \eqref{psi-eq} in terms of characteristic coordinates $(\al,\beta)$,
     \begin{align}
  \frac{\pt \psi}{\pt \beta}= \frac{\pt t}{\pt \beta} \cdot  H(x(\al,\beta),t(\al,\beta)).  \label{PsiBeta0}
  \end{align}
  Differentiating with respect to $\al$ yields
  \begin{equation}
  \dfr{\pt^2\psi}{\pt \al\pt\beta} =\dfr{\pt^2t}{\pt \al\pt \beta} \cdot H(\al,\beta) + \dfr{\pt t}{\pt \beta} \dfr{\pt H}{\pt \al}.
  \end{equation}
  Again, we use \eqref{second-tau} and  the fact that $\frac{\pt t}{\pt \beta}(0,\beta)$=0 to obtain
  \begin{equation}
  \dfr{\pt}{\pt\beta}\left[\dfr{\pt \psi}{\pt \al}(0,\beta)\right] = \dfr{1}{2c(0,\beta)}  \Theta(\beta) H(0,\beta),
  \end{equation}
 which provides by   integrating from $\beta_L$ to $\beta$
    \begin{align}
  \frac{\pt \psi}{\pt \alpha}(0,\beta)-\frac{\pt \psi}{\pt \alpha}(0,\beta_L)=\int_{\beta_L}^{\beta} \frac{1}{2c(0,\xi)} \cdot \Theta(\xi) \cdot H(0,\xi) d \xi.
  \label{psi-al-value}
  \end{align}
  Then there remain two issues unanswered: (i) one is about  the initial value $\frac{\pt\psi}{\pt \al}(0,\beta_L)$; (ii)  the other is the variation of $\psi$ in terms of the physical independent variables $(x,t)$.
 \vspace{0.2cm}

 \n (i)  {\em   The initial value $\frac{\pt\psi}{\pt \al}(0,\beta_L)$}. Note that
\begin{equation}
\dfr{\pt\psi}{\pt \al} =\dfr{\pt t}{\pt \al} \cdot \left[\dfr{\pt \psi}{\pt t} + (u-c) \dfr{\pt \psi}{\pt x}   \right],
\label{psi-al-t}
\end{equation}
and plug \eqref{psi-eq} into this identity.
 By setting $\beta=\beta_L$ we have
\begin{equation}
\dfr{\pt \psi}{\pt \al} (0,\beta_L)=\Theta(\beta_L) ( -2c_L\psi_L' +c_LK(\rho_L,s_L)s_L').
\label{psi-ini}
\end{equation}
\vspace{0.2cm}

\n {\em (ii) The return to the $(x,t)$-frame.} The combination of \eqref{psi-al-t} and   \eqref{psi-eq} for $\psi$ gives
\begin{equation}
2c \dfr{\pt \psi}{\pt x}(0,\beta) =c K(\rho,s)\dfr{\pt s}{\pt x}(0,\beta) -\Theta^{-1}(\beta) \dfr{\pt \psi}{\pt \al}(0,\beta).
\label{psi-al-x}
\end{equation}
We collect \eqref{psi-al-value}, \eqref{psi-ini} and \eqref{psi-al-x} to obtain
\begin{equation}
\begin{array}{rl}
2c\dfr{\pt \psi}{\pt x} (0,\beta) =& \d cK(\rho,s)\dfr{\pt s}{\pt x}(0,\beta) -\dfr{\Theta(\beta_L)}{\Theta(\beta)}( -2c_L\psi_L' +c_LK(\rho_L,s_L)s_L')\\[3mm]
&\d - \dfr{\Theta(\beta_L)}{\Theta(\beta)}\int_{\beta_L}^{\beta} \frac{1}{2c(0,\xi)} \cdot \dfr{ \Theta(\xi)}{\Theta(\beta_L)} \cdot H(0,\xi) d \xi.
\end{array}
\end{equation}
We continue using \eqref{psi-eq}  to obtain at $(0,\beta)$
\begin{equation}
\begin{array}{rl}
\dfr{\pt \psi}{\pt t} +u\dfr{\pt \psi}{\pt x} &  = -c\dfr{\pt \psi}{\pt x}+cK(\rho,s)\dfr{\pt s}{\pt x}\\[3mm]
&= \dfr{c}{2} \d K(\rho,s)\dfr{\pt s}{\pt x}(0,\beta) +\dfr{1}{2}\dfr{\Theta(\beta_L)}{\Theta(\beta)}( -2c_L\psi_L' +c_LK(\rho_L,s_L)s_L')\\[3mm]
&\d +\dfr{1}{2} \cdot\dfr{\Theta(\beta_L)}{\Theta(\beta)} \int_{\beta_L}^{\beta} \frac{1}{2c(0,\xi)} \cdot \dfr{ \Theta(\xi)}{\Theta(\beta_L)} \cdot H(0,\xi) d \xi.
\end{array}
\end{equation}
Denote the total (material)  derivative  by $D_0/Dt =\pt/\pt t+u\pt/\pt x$. Then we have
\begin{equation}
\dfr{D_0 u}{D t} +\dfr{1}{\rho c} \dfr{D_0 p}{Dt}  =d_L,
\label{u-p-rare}
\end{equation}
where $d_L$ takes ,
\begin{equation}
\begin{array}{rl}
d_L =&\dfr{c_L}{2} \d K(\rho,s) s_L' \Pi^{-2}(c;\beta,\beta_L) +\dfr{1}{2} \cdot \Pi(c;\beta,\beta_L)  ( -2c_L\psi_L' +c_LK(\rho_L,s_L)s_L')\\[3mm]
&\d +\dfr{1}{2}\Pi(c;\beta,\beta_L)  \int_{\beta_L}^{\beta} \frac{1}{2c(0,\xi)} \cdot \Pi(c;\xi,\beta_L) \cdot H(0,\xi) d \xi.
\end{array}
\label{d-L-1}
\end{equation}
 The integral in  $d_L$ can be either approximated numerically for very general cases (no explicit equation of state), or integrated out explicitly.
 For polytropic gases  $d_L$  is  (by denoting $\theta(\beta)=c(0,\beta)/c_L$),
\begin{equation}
d_L=\left[\dfr{1+\mu^2}{1+2\mu^2}
\left(\theta(\beta)\right)^{1/(2\mu^2)}+\dfr{\mu^2}{1+2\mu^2}
\left(\theta(\beta)\right)^{(1+\mu^2)/\mu^2}\right]T_LS'_L-c_L\left(\theta(\beta)\right)^{1/(2\mu^2)}
\psi'_L.
\label{d-L-2}
\end{equation}
\vspace{0.2cm}

\begin{prop} \label{prop-psi} Consider the curved rarefaction wave associated with $u-c$. The interaction of kinematic-thermodynamic variables can be described as
\begin{equation}
a_L\left( \dfr{D_0 u}{D t}\right)_* +b_L \left( \dfr{D_0 p}{D t}\right)_*=d_L, \ \ (a_L, b_L) =\left(1,\dfr{1}{\rho_*c_*}\right),
\end{equation}
where $d_L$ is given in \eqref{d-L-1} for a generic equation of state or \eqref{d-L-2} for polytropic gases.
\end{prop}

\section{Resolution of shocks}

Let $x=x(t)$ be a shock with speed $\sg=x'(t)$ and separate two states $\bu(x,t)$ in the wave front and $\bar\bu(x,t)$ in the wave back. This shock is defined by the Rankine-Hugoniot relations,
\begin{equation}
\begin{array}{l}
\sg=\dfr{\rho u-\bar\rho \b u}{\rho-\bar\rho},\\[3mm]
(\rho u-\bar\rho\bar u)^2 =(\rho-\bar\rho)(\rho u^2 +p-\bar\rho\bar u^2 -\bar p),\\[3mm]
   e(\rho, p) -e(\bar \rho,\bar p) + (\tau-\bar \tau)\cdot \dfr{p+\bar p}2=0.
\end{array}
\label{R-H}
\end{equation}
The second identity, the kinematic-thermodynamic relation, can be written as
\begin{equation}
(u-\bar u)^2  =\dfr{1}{\rho \bar \rho} (\rho-\bar \rho)(p-\bar p).
\label{R-H-2}
\end{equation}
 The last one is named the {\em Hugoniot relation}, defining the jump of thermodynamical quantities alone.  With the condition that
\begin{equation}
\dfr{\pt e}{\pt \rho } +\dfr{\tau-\bar \tau}2>0,  \mbox{ or } \dfr{\pt e}{\pt p}-\dfr{p+\bar p}{2\rho^2} >0,
\end{equation}
we express $\rho$ in terms of $p$,
\begin{equation}
\rho =H(p;\bar p, \bar \rho).
\end{equation}
 We substitute this into \eqref{R-H-2} and obtain
\begin{equation}
u=\bar u \pm \Phi(p; \bar \rho, \bar p),  \ \ \ \Phi(p;\bar p,\bar \rho):=\sqrt{\dfr{1}{\rho\bar \rho} (\rho-\bar \rho)(p-\bar p)}.
\end{equation}
Therefore the Rankine-Hugoniot relations comprise of a kinematic-thermodynamic relation, a (pure thermodynamic) Hugoniot relation and an identification of the  propagation speed,
\begin{equation}
\begin{array}{l}
\sg=\dfr{\rho u-\bar \rho\bar u}{\rho-\bar \rho},\\[3mm]
u=\bar u \pm \Phi(p; \bar \rho, \bar p),\\[3mm]
\rho- H(p;\bar p,\bar  \rho) =0.
\end{array}
\label{RH-u}
\end{equation}
The signs ``$\pm$" correspond to $u\pm c$",  respectively.  All the details can be found in \cite{Courant}. \vspace{0.2cm}

As a key part of the GRP solver,  we need to track the singularity. Inherently, we make differentiation along the shock trajectory $x=x(t)$. Denote
\begin{equation}
\dfr{D_\sg}{Dt }  = \dfr{\pt }{\pt t} +\sg\dfr{\pt }{\pt x}.
\end{equation}
We specify to the shock associated with $u+c$ and take the plus sign in \eqref{RH-u}.
Then we have
\begin{equation}
\dfr{D_\sg u}{Dt} = \dfr{D_\sg \bar u}{Dt}  + \dfr{\pt \Phi}{\pt p} \dfr{D_\sg p}{Dt} +  \dfr{\pt \Phi}{\pt \bar \rho} \dfr{D_\sg \bar \rho}{Dt} + \dfr{\pt \Phi}{\pt \bar p} \dfr{D_\sg \bar p}{Dt}.
\end{equation}
 Note that in smooth regions there hold
 \begin{equation}
\dfr{D_0\rho}{Dt} +\rho \dfr{\pt u}{\pt x} =0, \ \ \  \rho \dfr{D_0 u}{Dt} +\dfr{\pt p}{\pt x}=0, \ \ \ \ \dfr{D_0 p}{Dt} +\rho c^2 \dfr{\pt u}{\pt x}=0.
 \end{equation}
 The Lax-Wendroff methodology is adopted to replace the spatial derivatives of solutions in the wave front by the corresponding temporal derivatives, $\frac{\pt \bu}{\pt x} \rw \frac{\pt \bu}{\pt t}$; and replace the temporal derivatives in the wave back by the corresponding spatial derivatives $\frac{\pt \bar \bu}{\pt t}\rw\frac{\pt \bar \bu}{ \pt x}$.  Taking the limit along the shock trajectory $x=x(t)$ with
 \begin{equation}
 \begin{array}{ll}
\d  \lim_{t\rw 0+} \bar \bu(x(t)+0,t) =\bu_R,  \ \ &\d \lim_{t\rw 0+} \dfr{\pt \bar \bu}{\pt x}(x(t)+0,t)=\bu_R',\\[3mm]
\d   \lim_{t\rw 0+} \bar \bu(x(t)-0,t) =\bu_*,  \ \ &\d \lim_{t\rw 0+} \dfr{D_0\bar \bu}{Dt }(x(t)-0,t)=\left(\dfr{D_0\bu}{Dt}\right)_*,
 \end{array}
 \end{equation}
 we can express the temporal variation of solution in the intermediate region in terms of the local Riemann solution and the spatial variation of initial data from the right, as stated in the following proposition.
 \vspace{0.2cm}

\begin{prop}\label{prop-shock} Consider the shock associated with the characteristic field $u+c$. The temporal variation of solution in the intermediate region is described as
\begin{equation}
a_R\left(\dfr{D_0 u}{Dt}\right)_* +b_R\left(\dfr{D_0 p}{Dt}\right)_* =d_R,
\end{equation}
where the coefficients $a_R$, $b_R$ and $d_R$ are given in terms of the intermediate state $\bu_*$ and the initial data from the right,
\begin{equation}
\begin{array}{l}
a_R=1+\rho_{2*} \cdot (\sg-u_*)\cdot \dfr{\pt \Phi}{\pt p}(p_*;p_R,\rho_R),\\[3mm]
b_R=-\left[\dfr{1}{\rho_{2*}\cdot c_{2*}^2}\cdot (\sg-u_*)+\dfr{\pt \Phi}{\pt p}(p_*;p_R,\rho_R)\right],\\[3mm]
d_R=L_p^R \cdot  p_R'+L_u^R \cdot u_R'+L_{\rho}^R \cdot \rho_R', \\\\
\end{array}
\label{right-coef}
\end{equation}
and
\begin{equation}
\begin{array}{l}
 L_p^R= -\dfr{1}{\rho_R}+(\sg-u_R)\cdot \dfr{\pt \Phi}{\pt
 \bar  p}(p_*;p_R,\rho_R),\\[3mm]
 L_u^R=
\sg-u_R -\rho_R\cdot c_R^2 \cdot \dfr{\pt \Phi}{\pt  \bar p}(p_*;p_R,\rho_R)-\rho_R \cdot \dfr{\pt
\Phi}{\pt \bar \rho}(p_*;p_R,\rho_R),\\[3mm]
L_{\rho}^R= (\sg-u_R)\cdot \dfr{\pt
\Phi}{\pt \bar \rho}(p_*;p_R,\rho_R).
\end{array}
\end{equation}
\end{prop}

Once $(D_0 u/D t)_*$ and $(D_0 p/D t)_*$ are available, we can  derive the temporal variation of density in the intermediate region $(\pt \rho/\pt t)_*$ by tracking the shock trajectory $x=x(t)$.

\begin{prop}\label{prop-shock-density}  If the intermediate state is located between a contact discontinuity and the right-moving  shock associated with $u+c$, then we have

\begin{equation}
g_{\rho}^R \left(\dfr{\pt \rho}{\pt t}\right)_* +g_p^R
\left(\dfr{Dp}{Dt}\right)_*+g_u^R\left(\dfr{Du}{Dt}\right)_* = u_*
\cdot h_R, \label{rho-deri-shock}
\end{equation}
where  $g_{\rho}^R$, $g_p^R$, $g_u^R$ and $f_R$ are constant,
depending on the initial data (\ref{data}) in the right hand
side and the Riemann solution $R^A(0;\bu_L,\bu_R)$. They are expressed
in the following,
\begin{equation}
\begin{array}{l}
g_{\rho}^R = u_*-\sg,   \ \  g_p^R= \dfr{\sg}{c_{2*}^2
}-u_*H_1,  \ \ \ g_u^R=u_*\cdot  \rho_{2*}(\sg-u_*)\cdot H_1,\\
h_R= (\sg-u_R)\cdot H_2 \cdot p_R'+(\sg-u_R)\cdot H_3 \cdot\rho_R'
-\rho_R \cdot \left(H_2\cdot c_R^2+H_3\right)\cdot u_R'.
\end{array}
\end{equation}
 and  $H_i$, $i=1,2,3$, are
\begin{equation}
H_1=\dfr{\pt H}{\pt p}(p_*; p_R,\rho_R ), \ \ \ H_2 =\dfr{\pt
H}{\pt \b p}(p_*; p_R,\rho_R ), \ \ \ H_3 =\dfr{\pt H}{\pt \b
\rho}(p_*; p_R,\rho_R). \label{def-H}
\end{equation}
 \end{prop}

\vspace{0.2cm}

We emphasize that \eqref{RH-u} is independent of the specific equation of state.  Specified to the polytropic cases, we refer to \cite{Ben2006} for details.

\section{The GRP solver }\label{sec-thermo}

The GRP solver has two main versions: Acoustic version and nonlinear version.  The acoustic version applies to the case that  the waves  from each cell boundary are weak and so to the regions in which the flow is    smooth.  Of course, if one would like to use a simplified version, it is a choice. The second order ADER method is the acoustic version of GRP solver. If the waves from a cell boundary is very strong, the nonlinear GRP has to be used because the thermodynamic effect becomes significant.  We just provide the GRP solver when the cell interface $x=0$ is located inside the intermediate region or  the sonic case.  If all waves from the singularity point $(0,0)$ move to one side of the cell boundary $x=0$, $\bu_*$ and $(\pt \bu/\pt t)_*$ are just taken upwind.
\vspace{0.2cm}

\subsection{Acoustic GRP} As $\|\bu_L-\bu_R\|\ll 1$, the acoustic GRP solver is adopted.  Then $\bu_* = \bu_L=\bu_R$, and
$(\pt \bu/\pt t)_*$ is obtained by solving the linearized  system
\begin{equation}
\dfr{\pt u}{\pt t} +\bA(\bu_*) \dfr{\pt \bu}{\pt x} =G(x,\bu_*), \ \ \ \bA(\bu_*) =\dfr{\pt \bbf(\bu_*)}{\pt \bu}.
\end{equation}
Specified to the compressible Euler equations, we have the proposition.

 \begin{prop} [Acoustic case]  \label{thm-acoustic} As $\bu_L=\bu_*=\bu_R$
and $\bu_L'\neq \bu_R'$, we have the acoustic case. If $u_*-c_*<0$ and
$u_*+c_*>0$, then $(\pt u/\pt t)_*$ and $(\pt p/\pt t)_*$ can be
solved as
\begin{equation}
\begin{array}{c}
\left(\dfr{\pt u}{\pt t}\right)_*=-\dfr 12
\left[(u_*+c_*)\left(u_L'+\dfr{p_L'}{\rho_* c_*}\right )
+(u_*-c_*) \left(u_R'-\dfr{p_R'}{\rho_*
c_*} \right)\right],\\\\
 \left(\dfr{\pt p}{\pt t}\right)_*=-\dfr {\rho_*c_*}2
\left[(u_*+c_*)\left(u_L'+\dfr{p_L'}{\rho_* c_*}\right )
-(u_*-c_*) \left(u_R'-\dfr{p_R'}{\rho_*
c_*} \right)\right].\\
\end{array}
\label{acoustic1}
\end{equation}
Then the quantity $(\pt \rho/\pt t)_*$ is calculated from the
equation of state $p=p(\rho, s)$,
\begin{equation}
\left(\dfr{\pt\rho}{\pt t}\right)_*=\left\{ \begin{array}{ll}
 \dfr{1}{c_*^2} \left[
\left(\dfr{\pt p}{\pt t}\right)_* +u_* \left( p_L'-c_*^2\rho_L'
\right) \right]\ \ \ & \mbox{ if } u_*=u_L=u_R>0,\\\\
 \dfr{1}{c_*^2} \left[
\left(\dfr{\pt p}{\pt t}\right)_* +u_* \left( p_R'-c_*^2\rho_R'
\right) \right], & \mbox{ if }u_*=u_L=u_R<0.
\end{array}
\right. \label{acoustic2}
\end{equation}
\end{prop}

 \vspace{0.2cm}

 \subsection{Nonlinear GRP}  As $\|\bu_L-\bu_R\|\gg 1$, the nonlinear GRP solver has to be applied.  The numerical example in the present paper shows the importance of the nonlinear GRP solver.   The nonlinear GRP solver for compressible fluid flows comprises of two parts:

 \begin{enumerate}
 \item[(i)] {\em Kinematic-thermodynamic relation. }  The material derivatives of $u$ and $p$ are obtained by solving the linear algebraic system
 \begin{equation}
 \begin{array}{l}
 a_L\left( \dfr{D_0u}{Dt}\right) _*+ b_L \left( \dfr{D_0p}{Dt}\right) _*=d_L,\\[3mm]
 a_R\left( \dfr{D_0u}{Dt}\right) _*+ b_R \left( \dfr{D_0p}{Dt}\right) _*=d_R,
 \end{array}
 \label{GRP-nonlinear}
 \end{equation}
 where $(a_L,b_L,d_L)$ and $(a_R, b_R, d_R)$ are given in terms of local Riemann solution $R^A(x/t;\bu_L, \bu_R)$ and the corresponding initial data.
 \vspace{0.2cm}

 \item[(ii)] {\em} {The temporal  variation of pure thermodynamical quantities.} Once $(D_0 u/Dt)_*$ and $(D_0 p/Dt)_*$ are available, we can determine other thermodynamic quantities such as the density.  If $x=0$ is located between the rarefaction wave and the contact discontinuity, the Gibbs relation or the equation of state, $p=p(\rho,s)$, is used directly,
 \begin{equation}
 dp =c^2 d\rho +\dfr{\pt p}{\pt s} ds.
 \end{equation}
 Otherwise, if $x=0$ is located between the shock and the contact discontinuity, \eqref{rho-deri-shock} is used  to obtain $(\pt \rho/\pt t)_*$.
 \vspace{0.2cm}

 \end{enumerate}

 From the  derivation of \eqref{GRP-nonlinear}, the thermodynamics  is included in the coefficients $d_L$ and $d_R$ to
 exert on $(D_0 u/Dt)_*$ and $(D_0 p/Dt)_*$ and then on $(\pt \rho/\pt t)_*$. In the end the thermodynamic effect is exhibited in the numerical fluxes to influence numerical solutions.

\section{ Numerical demonstration of thermodynamical effect}\label{sec-num}

In order to show the importance of thermodynamic effects, we choose the example from \cite{Tang-Liu}, which is equivalent to the Leblanc problem. The initial data is of Riemann-type,
\begin{equation}
(\rho, u,p)(x,0) = \left\{
\begin{array}{ll}
(10^4, 0, 10^4), \ \ \ &x<0,\\[3mm]
(1,0,1), &x>0.
\end{array}
\right.
\end{equation}
The polytropic index $\gm=1.4$.
The solution consists of a very strong left-going rarefaction wave, a relatively weak right-moving  shock, and a contact discontinuity in the middle, besides constant states. The rarefaction wave spans the density from $10^4$ to
$80$, while the density just jumps from around $1$ to $8$ across the shock. If the rarefaction wave is displayed with 100 grid points, the jump of density values at neighboring grid points is about $100$.  Hence,
the difference of density (or pressure) values at neighboring grid points inside the rarefaction wave is much larger than that near the shock. Hence  any possible poor resolution of the rarefaction wave may cause the numerical inaccurate  capturing of the shock (including the wave speed and the strength).  The following numerical demonstration is carried out with CFL number $0.5$ for first order schemes and $0.32$ for second order schemes.  All figures are displayed with $66$ points.
\vspace{0.2cm}

We start the numerical demonstration with first order methods, as shown in Figures \ref{Fig-first-order}.  The first order solvers we choose are the popularly-used exact Riemann solver (Godunov scheme), HLLC solver and Roe solver with entropy fix.  We carry out the computations using different grid points: $200$ points, $1000$ points and $10^4$ grid points, respectively.  It is observed that there are large disparities of numerical solutions from the exact one even with $10^4$ grid points. The reason, we think,  is the following. Almost all  first order schemes  assume that the flow is uniform at every time step inside each computational cell, which results in the fact that the associated rarefaction wave is isentropic. The uniformity of entropy leads to the fact that the entropy variation is difficult to be included in numerical fluxes unless it can be made up through the jumps from cell boundaries.
\vspace{0.2cm}

\begin{figure}[!htb]
\subfigure[$200$ grid points ]{
  \includegraphics[width=2.5in]{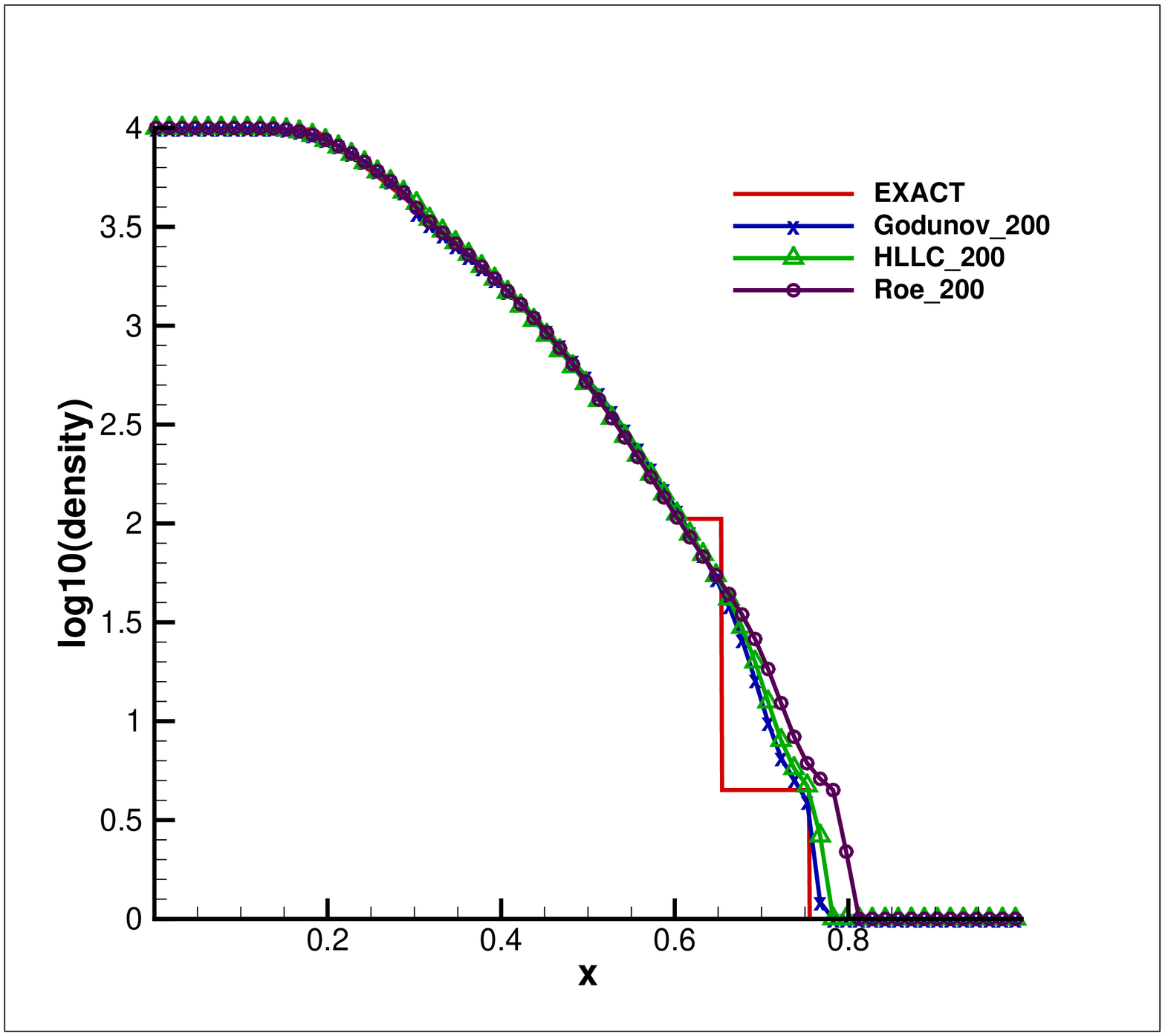}}
\hspace{1cm}
\subfigure[$1000$ grid points]{  \includegraphics[width=2.5in]{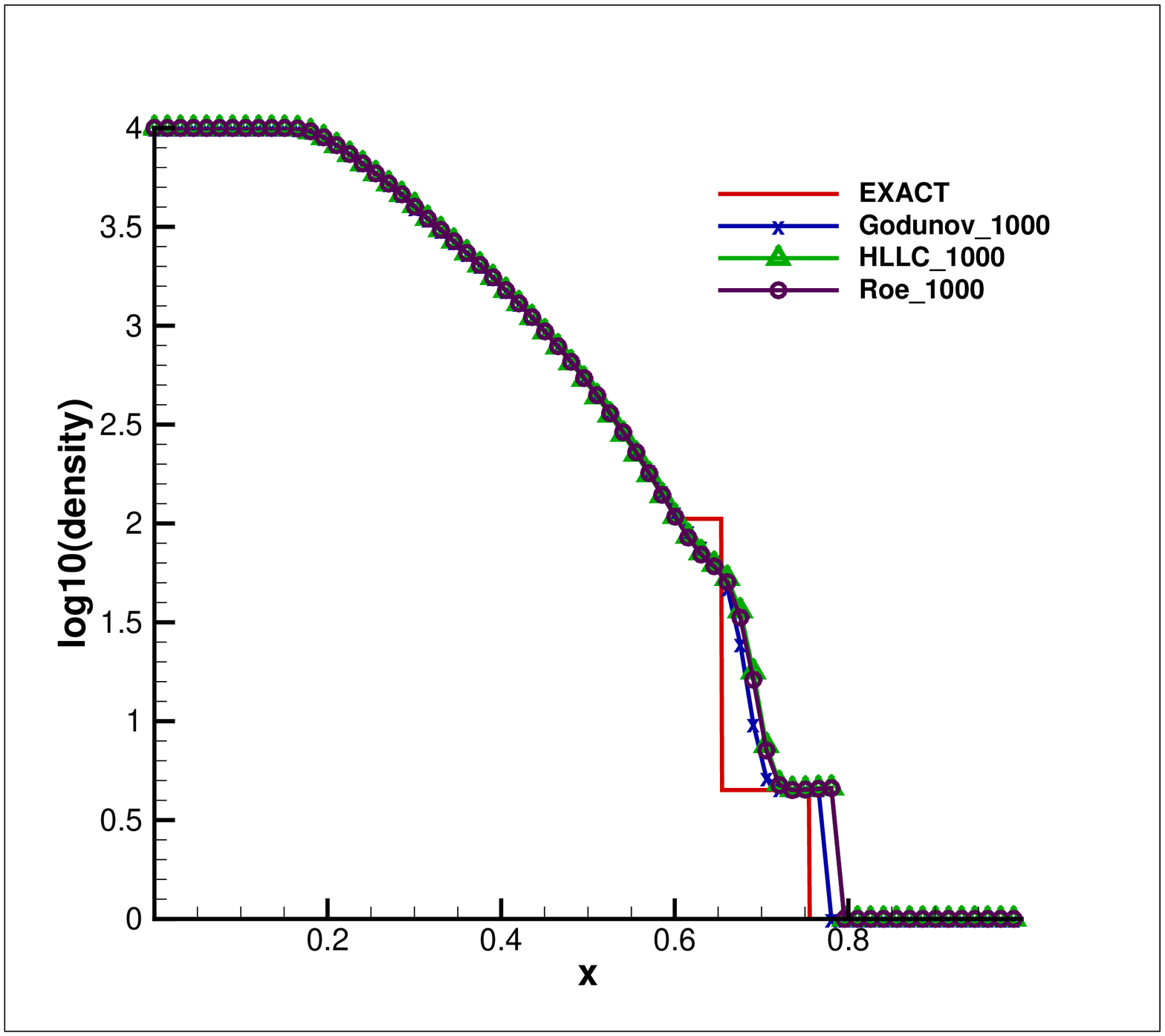} }
\subfigure[$10000$ grid points]{ \includegraphics[width=2.5in]{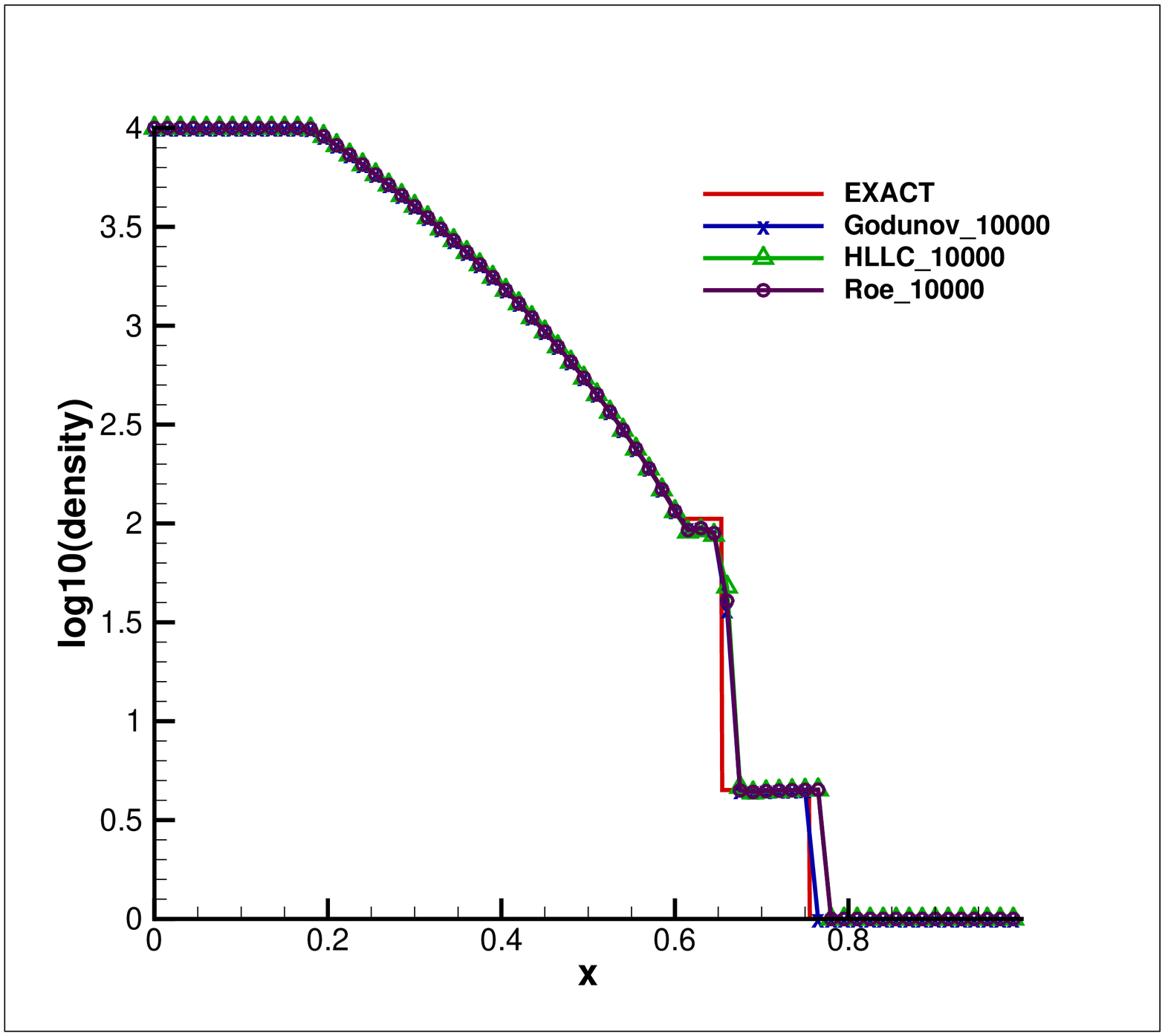}
 } \hspace{1cm}
  \subfigure[Zoomed solution of (c)]{
  \includegraphics[width=2.5in]{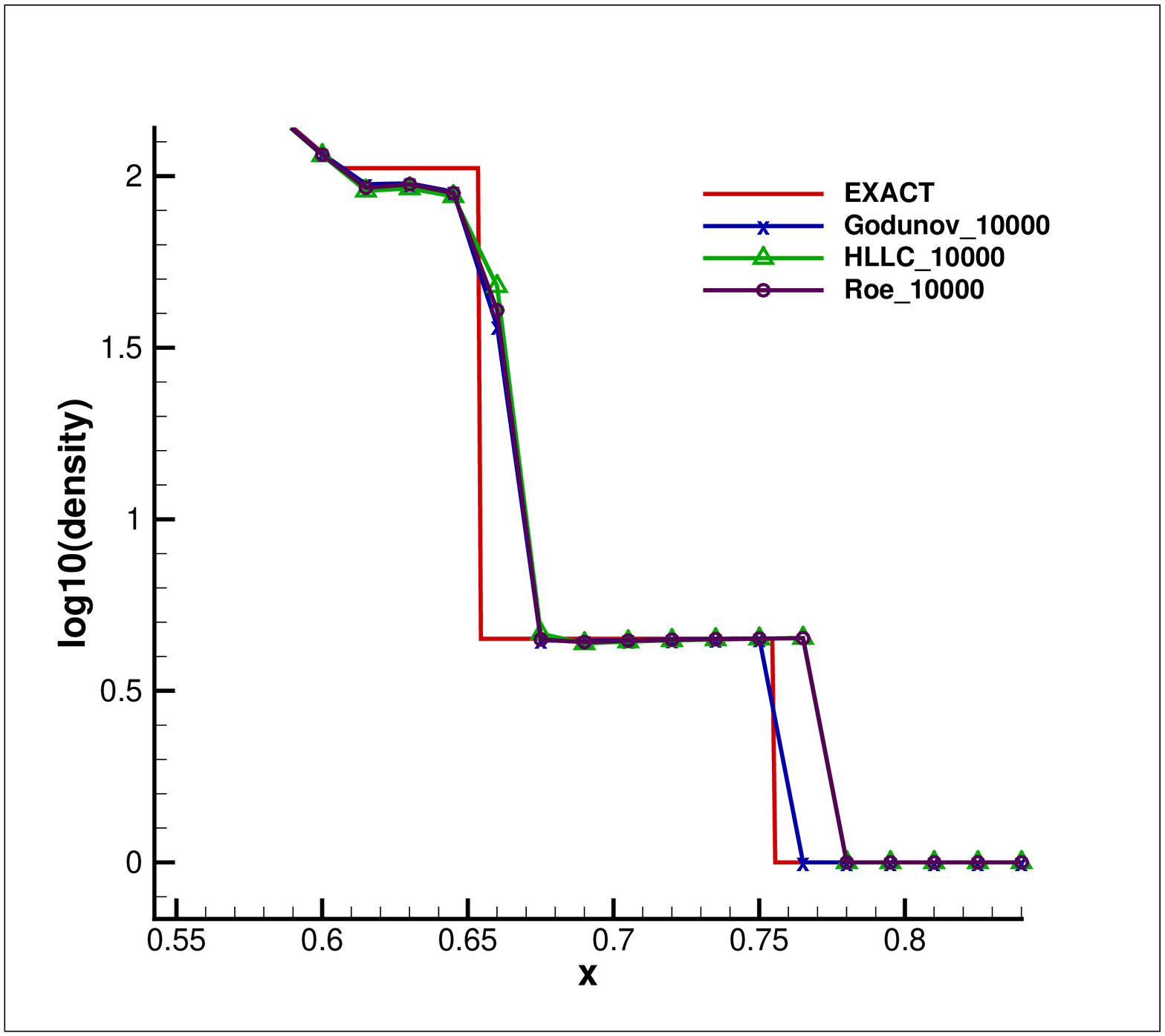}  }

  \caption{The numerical solutions computed by the  first order schemes (with the exact, HLLC, Roe Riemann solvers) are compared with the exact solution (only 66 cells are shown).}\label{Fig-first-order}
\end{figure}

We proceed to look into the situation that schemes have accuracy of second order in space and first order in time, shown in Figure \ref{Fig-second-first}.  At each time step, the MUSCL type data reconstruction is adopted, but only first order flux solvers are used. It is observed that there is no significant improvement at all and even oscillations are present near the contact discontinuities.  The large disparities of numerical solutions from the exact solution result from the inaccurate resolution of rarefaction waves and the presence  oscillations are due to the inconsistency of spatial and temporal accuracy.

\vspace{0.2cm}

\begin{figure}[htb]
  \subfigure[$1000$ grid points ]{
  \includegraphics[width=2.5in]{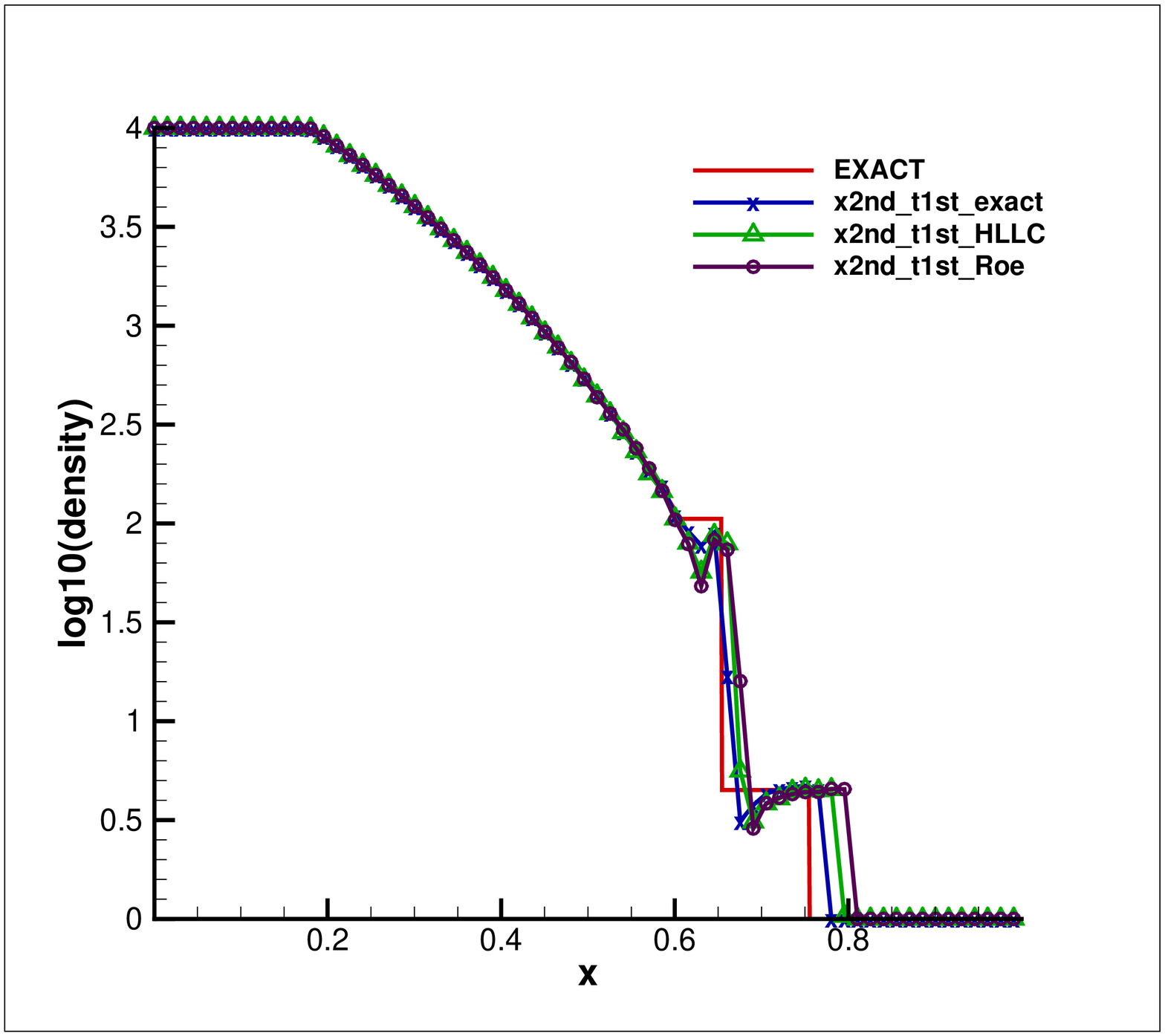}}
\hspace{1cm}
\subfigure[Zoomed solution]{  \includegraphics[width=2.5in]{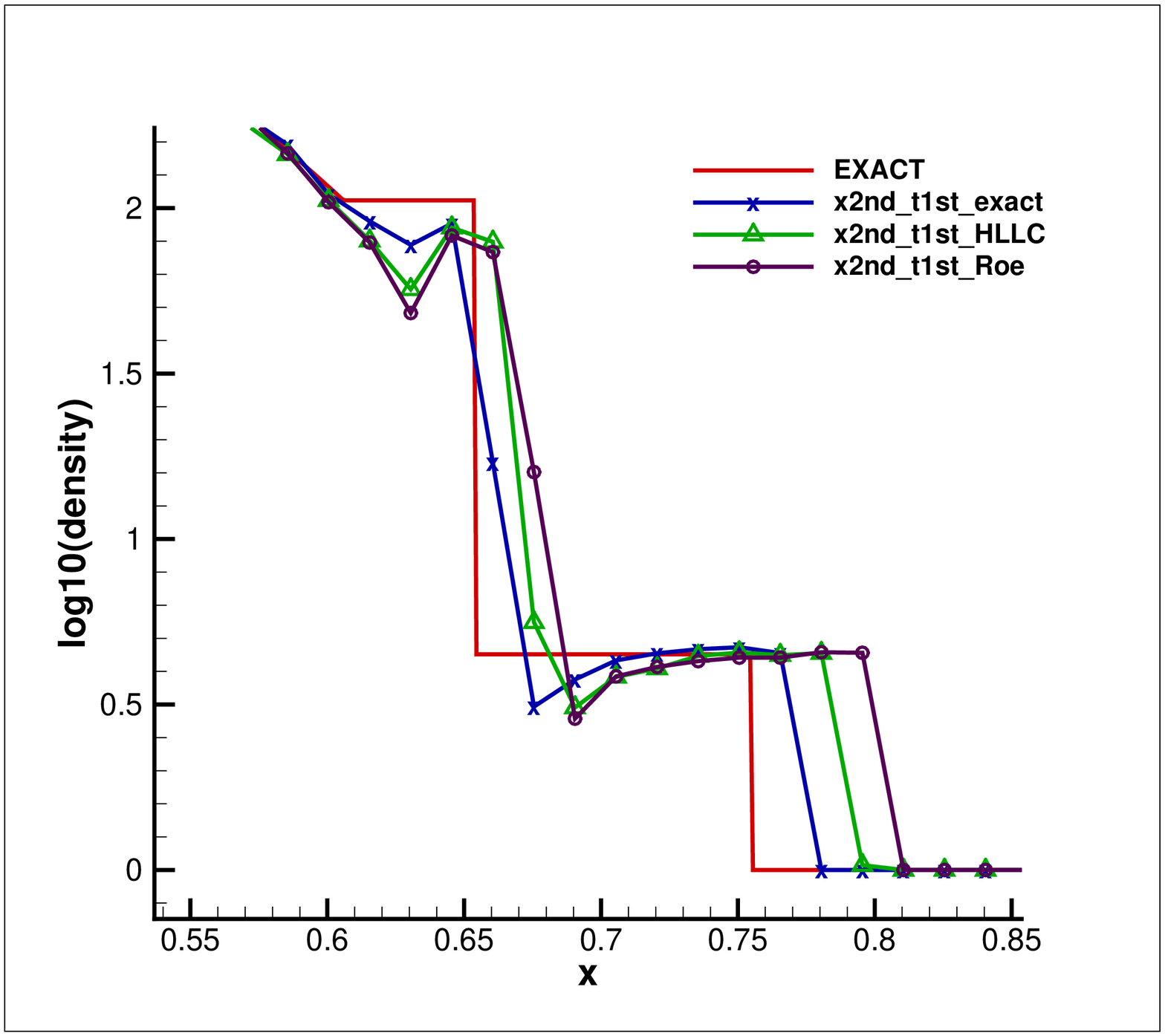} }
  \caption{The numerical solutions computed by the schemes with second order in space and first order in time (with the exact, HLLC, Roe Riemann solvers and $1000$ cells) are compared with the exact solution (only 66 cells are shown). }
  \label{Fig-second-first}
\end{figure}

\begin{figure}
  \subfigure[$1000$ grid points ]{
  \includegraphics[width=2.5in]{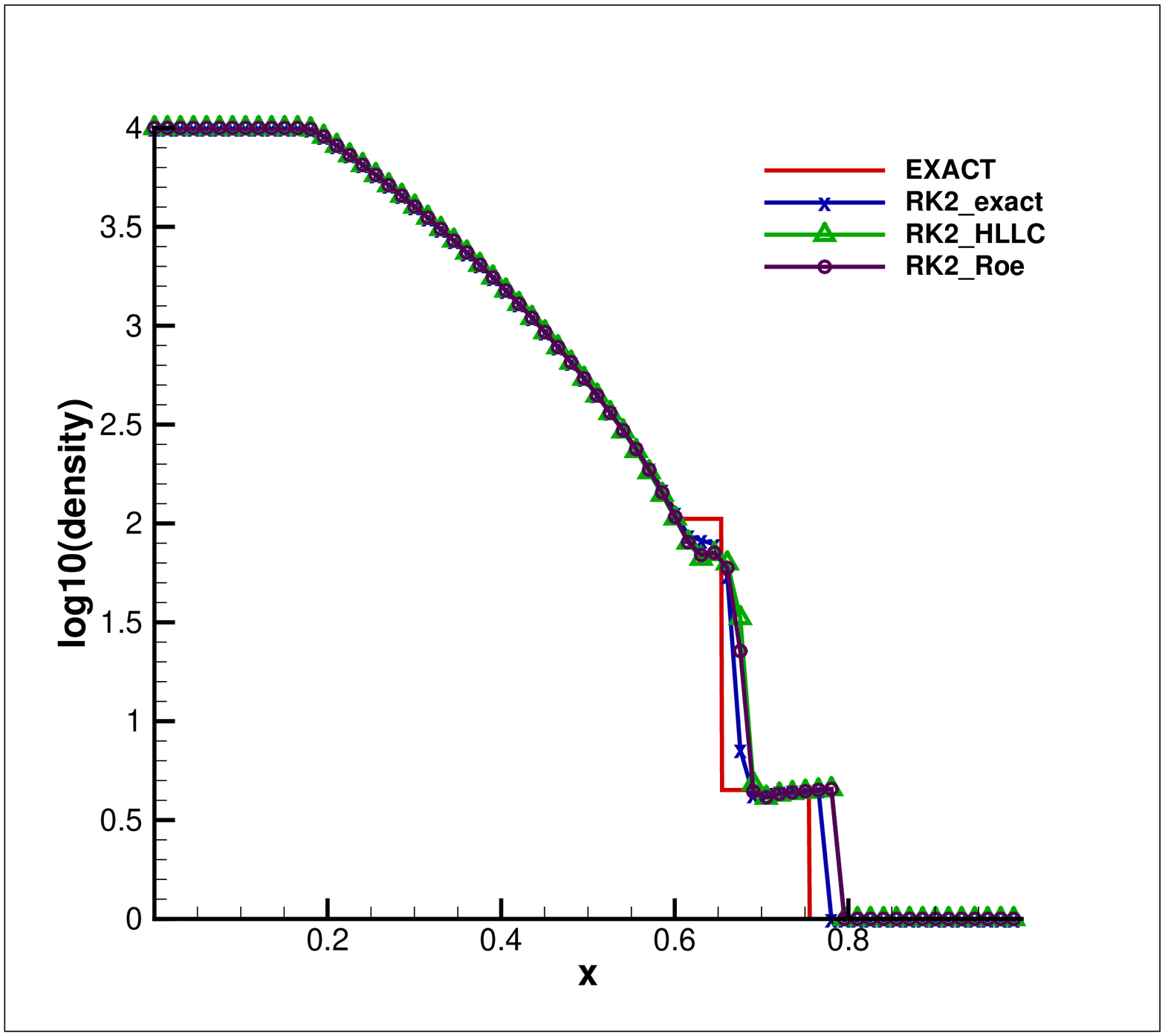}}
\hspace{1cm}
\subfigure[Zoomed solution]{  \includegraphics[width=2.5in]{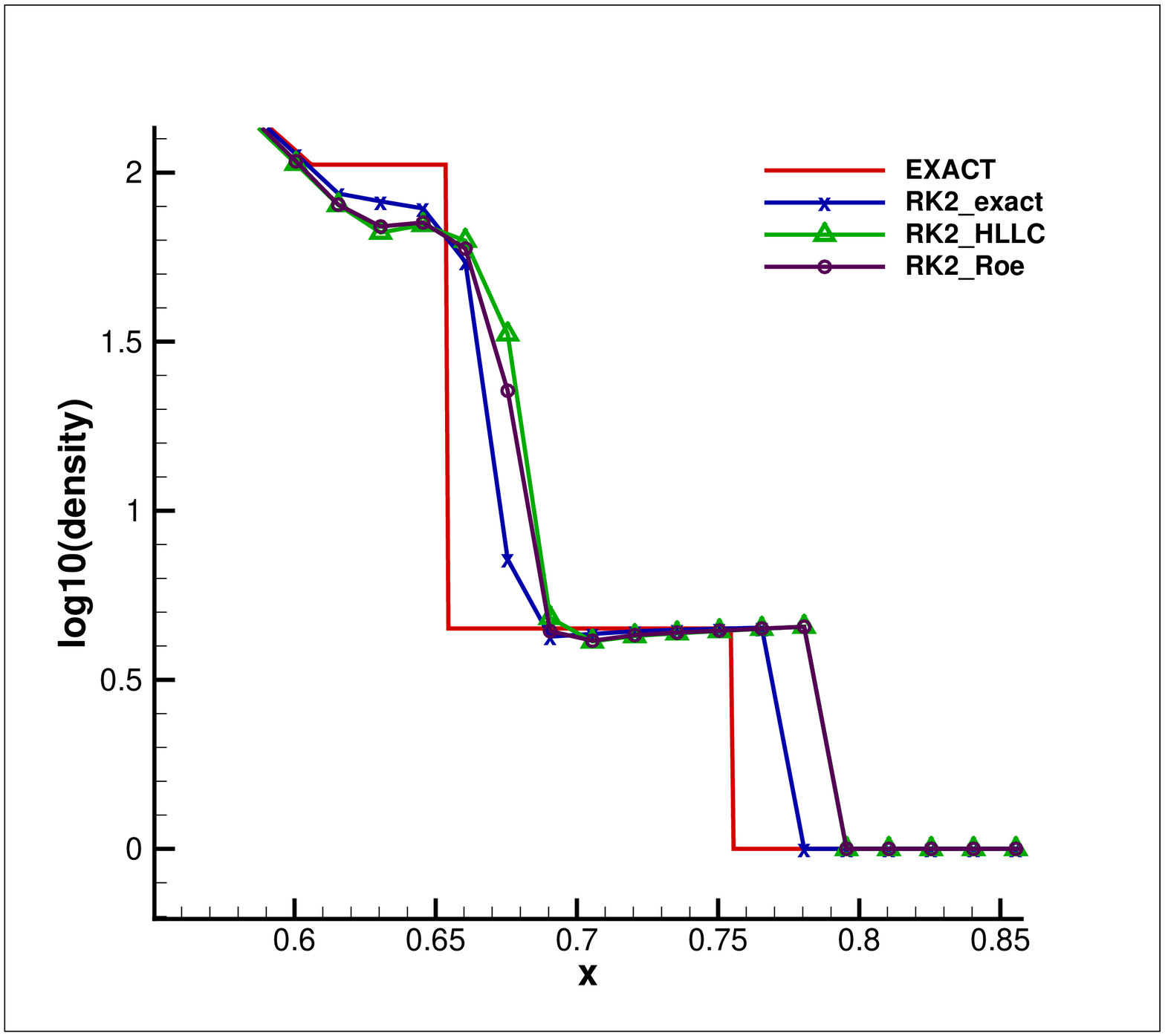} }
  \caption{The numerical solutions computed by the second order R-K schemes (with the exact, HLLC, Roe Riemann solvers and 1000 cells) are compared with the exact solution (only 66 cells are shown). }
\label{Fig-second-RK}
\end{figure}

In Figure \ref{Fig-second-RK}, we keep the MUSCL type data reconstruction at each time step, but increase the temporal accuracy to second order accuracy.  We first use the two-stage Rung-Kutta temporal iteration with first order flux solvers as building blocks. The result is improved a little bit but not prominently.   Hence we make a trial to execute  a one-stage method with acoustic GRP solver (consistent with   ADER \cite{Toro-ADER}).   The result is displayed in Figure \ref{Fig-second-acoustic}. The result is not satisfactory until the grid points are taken to be $10^4$, and even the oscillations near the edge of the rarefaction wave are more severe.
\vspace{0.2cm}

\begin{figure}
  \subfigure[Acoustic GRP solver with different grid points ]{
  \includegraphics[width=2.5in]{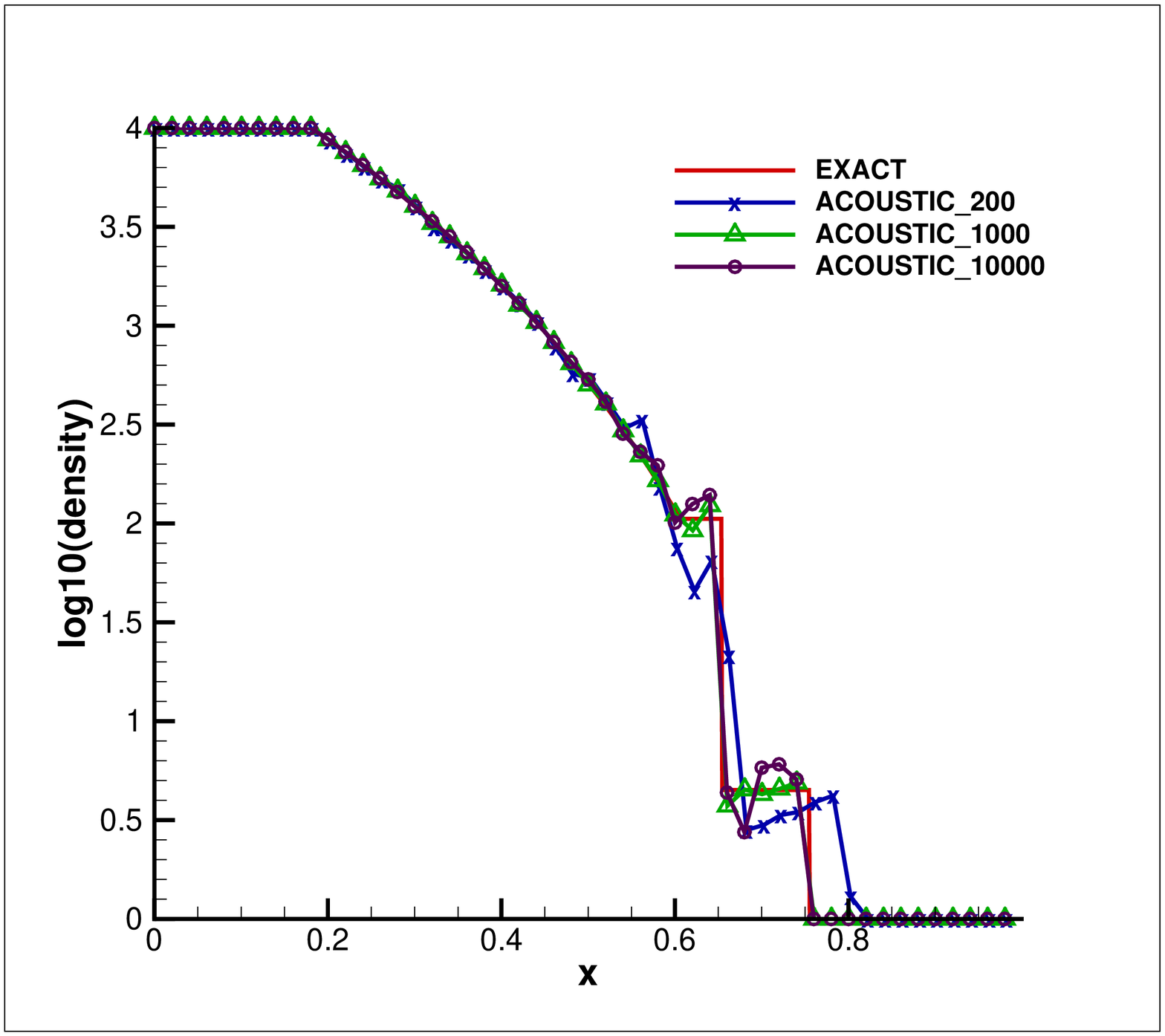}}
\hspace{1cm}
\subfigure[Zoomed solution]{  \includegraphics[width=2.5in]{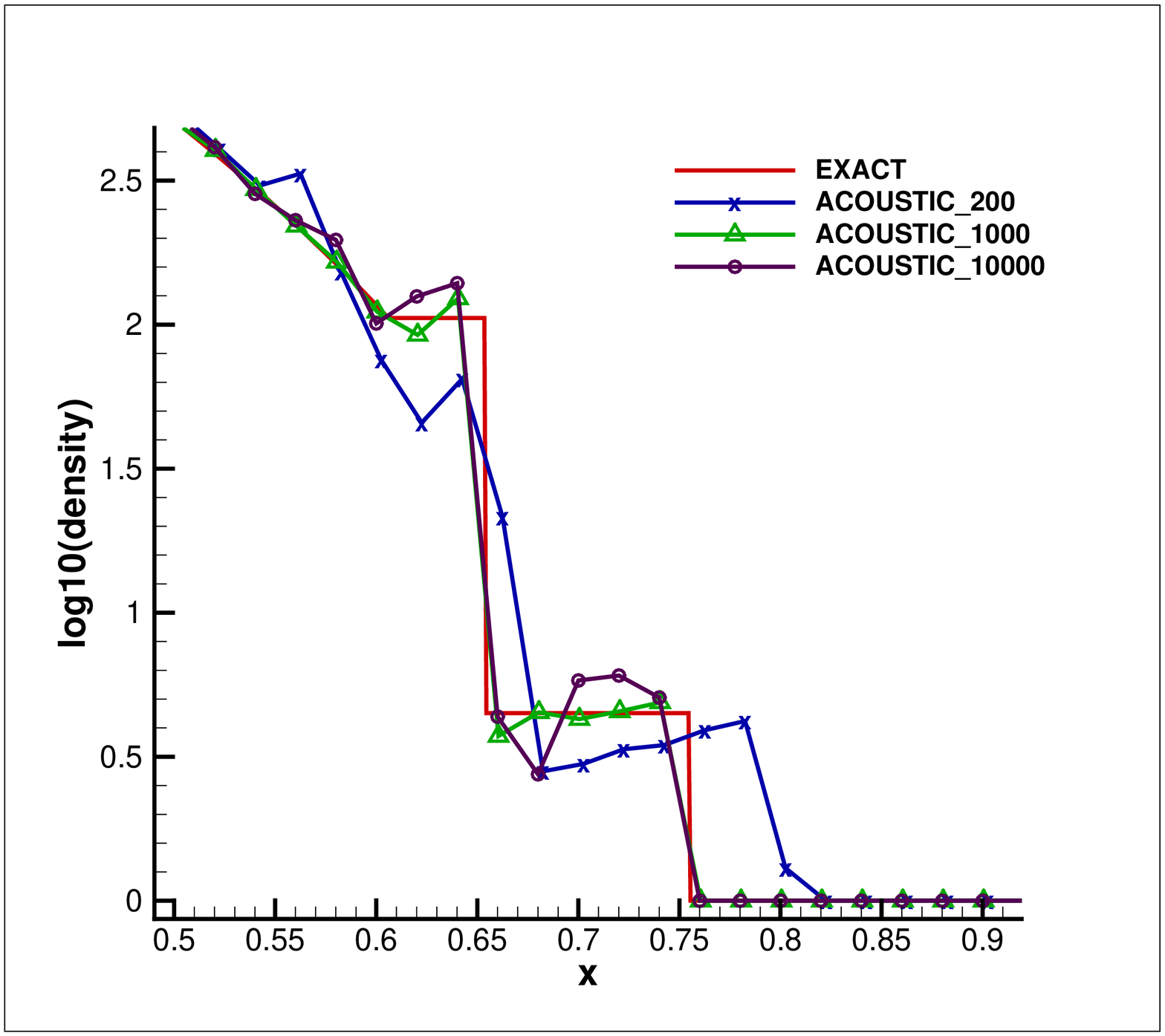} }
  \caption{The numerical solutions computed by the second order acoustic GRP solver (with the exact Riemann solver) are compared with the exact solution (only 66 cells are shown). }
  \label{Fig-second-acoustic}
\end{figure}

Finally, we use the nonlinear GRP solver in the strong wave regions to simulate this problem with $100$, $200$ and $300$ grid points, respectively.   The result is displayed in Figure \ref{Fig-GRP}.  We can see that the computation with $300$ grid points can effectively cope with the resolution of strong waves.

\begin{figure}[!ttb]
  \subfigure[GRP with relatively small number of  grid points ]{
  \includegraphics[width=2.5in]{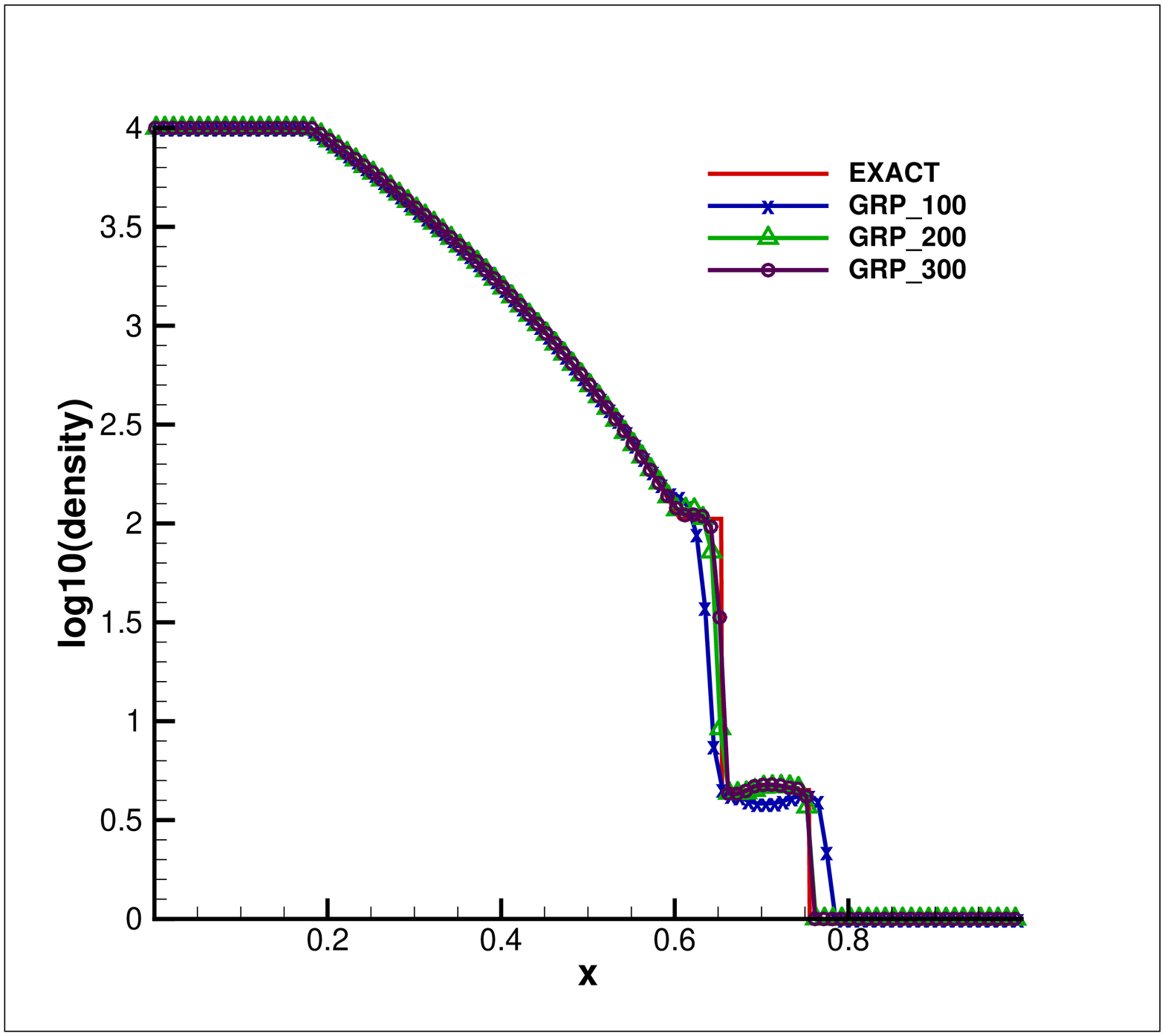}}
\hspace{1cm}
\subfigure[Zoomed solution]{  \includegraphics[width=2.5in]{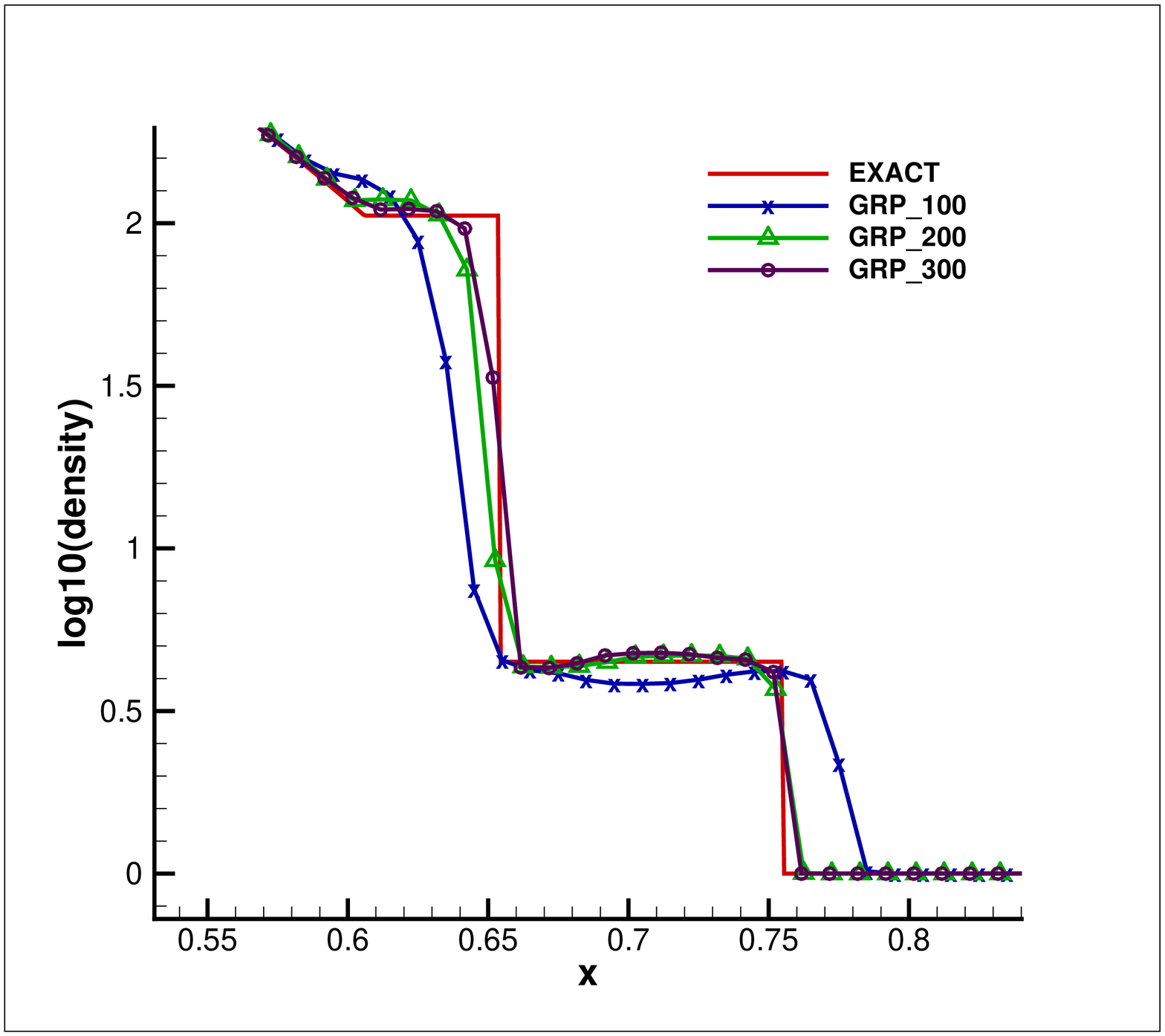} }
\caption[small]{The GRP simulation (only 100 cells are shown)}
\label{Fig-GRP}

\end{figure}


%
%

 \section{More remarks}

 In this paper, we refine the GRP solver and highlight the thermodynamic effect on the simulation of strong waves. Although this work is done just for compressible Euler equations, the conclusion is of general significance.  As for the use of GRP solver and its extension,   some remarks are in order.
 \vspace{0.2cm}

 \begin{enumerate}
 \item[(i)]  There are two versions of the GRP solver: Acoustic and nonlinear.  In smooth regions of flows or weak waves, just the acoustic solver is needed. However, once nonlinear waves become strong or thermodynamic effect becomes significant, the nonlinear GRP solver has to be used. It turns out that in our simulations, the acoustic solver is used for most of the computational time, and the nonlinear GRP solver works just near strong waves.

 \item[(ii)] We just refine the 1-D GRP solver with a single physical effect in the present paper.  As more physical factors are included, the GRP solver can be derived similarly. To be more precise, these factors take effect on the kinematical-thermodynamic variables and the equation  \eqref{psi-eq} is replaced by
 \begin{equation}
 \dfr{\pt \psi}{\pt t} +(u+c)\dfr{\pt\psi}{\pt x} = cK\dfr {\pt s}{\pt x} +G_1+G_2 +\cdots,
 \end{equation}
 where $G_1$, $G_2$, $\cdots$, represent the aforementioned physical effects. It turns out that in \eqref{u-p-rare} $d_L$ is replaced by $\td d_L$,
 \begin{equation}
 \td d_L =d_L + \sum_{i} d_{iL},  \ \ \ d_{iL}= \dfr{1}2 \dfr{\Theta(\beta_L)}{\Theta(\beta)} \int_{\beta_L}^{\beta} \dfr{1}{2c(0,\xi)} \dfr{\Theta(\xi)}{\Theta(\beta_L)} G_i(0,\xi)d\xi.
 \end{equation}
 The integrals for $d_{iL}$ can be evaluated either precisely (if possible) or numerically.  In parallel, $d_R$  should include other physical effects through the Lax-Wendroff methodology.
 \vspace{0.2cm}

 \item[(iii)]  For a specific equation of state, the Gibbs relation has the corresponding implication.  For example, in chemical thermodynamics \cite{Eu},   the internal energy comprises of a static energy $e_0$, the chemical potential $e_i$
 \begin{equation}
 e=e_0 + e_i,
 \end{equation}
 Then the Gibbs relation becomes
 \begin{equation}
 Tds =de_0 +d e_i -\dfr{p}{\rho^2}d\rho.
 \end{equation}
 Recall that in the computation of entropy variation, any initial variation   in \eqref{en-ini} is reflected in the corresponding numerical fluxes, and so does   chemical potential  $de_i$. This is an example how the GRP solver includes additional effects in practical applications.

 \vspace{0.2cm}

 \item[(iv)] The role of GRP solver for second order schemes is just like that of the exact Riemann solver for first order schemes, but it has its own significance. We can regard the schemes based on the GRP solver as discontinuous Lax-Wendroff schemes.  The spatial-temporal coupling can effectively absorb the full (physical)  information provided by the governing equations, which is the spirit of the Cauchy-Kowalevski methodology.

 \end{enumerate}

\section*{Acknowledgement}

Jiequan Li is supported by by NSFC (No. 11371063, 91130021),  the doctoral program from the Education Ministry of China (No. 20130003110004). Yue Wang is supported by NSFC (No. 11501040).

\newpage
\centerline{TABLE I: Basic notations}

\begin{center}
\begin{tabular}{ll}
\hline \cline{1-2}
 Symbols & Definitions\\
 \cline{1-2}
$\rho$, $u$, $p$, $s$\ \ \ \ \ \ \ \ \ \  & density, velocity, pressure, entropy\cr
$\phi$, $\psi$ & kinematic-thermodynamical variables\\
 $\bu_L, \bu_R$  \ \ \ \ & $\lim \bu(x,0)$ as
$x\rw 0_-$,  $x\rw 0_+$\\
 $\bu'_L$, $\bu'_R$  & constant slopes $\dfr{\pt \bu}{\pt
x}$ for $x<0$,   $x>0$\\
$R^A(\cdot; \bu_L,\bu_R)$ & solution of the Riemann problem subject to data $\bu_L$, $\bu_R$\\
$\bu_*$ & $R^A(0;\bu_L,\bu_R)$ \\
$\bu_{1*}$, $\bu_{2*}$ & the value of $\bu$ to the left, the right of contact
discontinuity\\
$\bu_-(x,t)$, $\bu_+(x,t)$ &  the solution in the left,   the
right\\
 $\left(\dfr{\pt \bu}{\pt t}\right)_*$ & $\dfr{\pt
\bu}{\pt t}(x,t)$
at $x=0$ as  $t\rw 0_+$\\
$D_0\bu/Dt$ & the material derivative of $\bu$, $\dfr{\pt \bu }{\pt t}+u\dfr{\pt \bu }{\pt x}$\\
$(D_0\bu/Dt)_*$ & the limiting value of $D_0\bu/Dt$ at $x=0$ as $t\rw
0_+$\\
$u-c$, $u$, $u+c$ & three eigenvalues\\
$\beta$, $\al$ & two characteristic coordinates \\
$\sg_L$, $\sg_R$ & shock speed at time zero, corresponding to
$u-c$, $u+c$\\
 $\mu^2=\dfr{\gm-1}{\gm+1}$ & $\gm>1$ the
polytropic index,
$\gm=1.4$ for air\\
\cline{1-2}
\end{tabular}
\end{center}

\end{document}